\documentclass [12pt,a4paper]{amsart}

\usepackage{amssymb,amsmath,amsthm, amscd}

\usepackage{url}

\begin{document}

	\newtheorem{theorem}{Theorem}[section]
	\newtheorem{lemma}[theorem]{Lemma}
	\newtheorem{proposition}[theorem]{Proposition}
	
	\newtheorem{statement}[theorem]{Statement}
	\newtheorem{conjecture}[theorem]{Conjecture}
	\newtheorem{corollary}[theorem]{Corollary}
	
	\newtheorem*{maintheorem}{Theorem A}

	\newtheoremstyle{neosn}{0.5\topsep}{0.5\topsep}{\rm}{}{\bf}{.}{ }{\thmname{#1}\thmnumber{ #2}\thmnote{ {\mdseries#3}}}
	\theoremstyle{neosn}
	\newtheorem{remark}[theorem]{Remark}
	\newtheorem{definition}[theorem]{Definition}
	\newtheorem{example}[theorem]{Example}

	\numberwithin{equation}{section}

	\newcommand{\Aut}{\,\mathrm{Aut}\,}
	\newcommand{\Inn}{\,\mathrm{Inn}\,}
	\newcommand{\End}{\,\mathrm{End}\,}
	\newcommand{\Out}{\,\mathrm{Out}\,}
	\newcommand{\Hom}{\,\mathrm{Hom}\,}
	\newcommand{\ad}{\,\mathrm{ad}\,}
	\newcommand{\SLn}{$\rm sl(n)$}
	\newcommand{\calL}{{\mathcal L}}
	\renewenvironment{proof}{\noindent \textbf{Proof.}}{$\blacksquare$}
	\newcommand{\GL}{\,\mathrm{GL}\,}
	\newcommand{\SL}{\,\mathrm{SL}\,}
	\newcommand{\SU}{\,\mathrm{SU}\,}
	\newcommand{\PGL}{\,\mathrm{PGL}\,}
	\newcommand{\PSL}{\,\mathrm{PSL}\,}
	\newcommand{\PSO}{\,\mathrm{PSO}\,}
	\newcommand{\SO}{\,\mathrm{SO}\,}
	\newcommand{\UN}{\,\mathrm{QSym}\,}
	\newcommand{\SUN}{\,\mathrm{SQ}\,}
	\newcommand{\Sp}{\,\mathrm{Sp}\,}
	\newcommand{\Spin}{\,\mathrm{Spin}\,}
	\newcommand{\PSp}{\,\mathrm{PSp}\,}
	\newcommand{\UT}{\,\mathrm{UT}\,}
	\newcommand{\Ann}{\,\mathrm{Ann}\,}
	\newcommand{\Exp}{\,\mathrm{Exp}\,}
	\newcommand{\Rad}{\,\mathrm{Rad}\,}
	\newcommand{\charr}{\mathrm{char}\,}
	\newcommand{\tr}{\mathrm{tr}\,}
	\newcommand{\diag}{{\rm diag}}

	\font\cyr=wncyr10 scaled \magstep1%

\title{Automorphisms of twisted Chevalley groups of type
${}^2E_6$ over local rings}

\author{Elena I. Bunina}
\author{Deep H. Makadiya}
\address{Department of Mathematics,
Bar--Ilan University, 5290002 Ramat Gan, Israel}
\date{}


\begin{abstract}
We classify abstract automorphisms of the elementary and full twisted
Chevalley groups of type ${}^2E_6$ over commutative local rings with
involution and $2$ invertible, for both isogeny types.
Every automorphism is standard, with no invertibility assumption on $3$.
In the adjoint case, it has a unique expression as a composition of
conjugation by an element of the full twisted group and a ring
automorphism commuting with the involution.
For the simply connected form, the full group equals its elementary
subgroup, and the factors are strictly inner, diagonal, and ring
automorphisms.
The main step replaces a linear change of basis by a group element.
Invariant alternating tensors determine the transported Chevalley
bracket. We correct this bracket over a faithfully flat local
extension and descend the resulting conjugating element by uniqueness.
Root-subgroup relations then recover the coefficient automorphism.
\end{abstract}

\maketitle

\section{Introduction}

This is the fourth paper in our series on automorphisms and
isomorphisms of twisted Chevalley groups over commutative rings.
The preceding papers treat the unitary types ${}^2A_\ell$,
$\ell\geq5$, the low-rank types ${}^2A_3$ and ${}^2A_4$,
and the orthogonal types ${}^2D_\ell$, $\ell\geq4$;
see \cite{BuninaMakadiya2A,BuninaMakadiya2A34,BuninaMakadiya2D}.
Here we treat the exceptional type ${}^2E_6$ over local rings.
All automorphisms are standard under the sole invertibility
assumption $2\in R^\times$.

The automorphisms are abstract: no preservation of root subgroups
or compatibility with the matrix realization is assumed.
Over fields, the standard descriptions are developed in
\cite{SteinbergLectures,Humphreys1969}.
Over local rings, the residual classification must be supplemented
by control of the congruence kernel and of root elements with arbitrary
parameters; see \cite{BuninaLocalADE,Bunina2012,Bunina2024}.
We use the construction and structural properties of twisted groups
from \cite{AbeTwisted,SuzukiLocalTwisted,SuzukiNormality,GMNormalSubgroups}.
The type-specific step is to recover the Chevalley Lie bracket
after an initial matrix normalization.

\medskip
\noindent\textbf{The groups and the main result.}
Let $R$ be a commutative local ring with identity, let
$J=\operatorname{Jac}(R)$, and put $k=R/J$.
Assume
\[
 2\in R^\times,\qquad
 \theta\in\operatorname{Aut}(R),\qquad
 \theta^2=\operatorname{id}_R.
\]
Write $\overline r=\theta(r)$ and
\[
 R_0=R^\theta=\{r\in R\mid\overline r=r\}.
\]
The involution may be trivial, and a nontrivial $\theta$ may act
trivially on $k$.
The ring need not be Noetherian, reduced, complete, or Henselian;
in particular, residue characteristic three is included.

Let $\tau$ be the nontrivial pinned graph automorphism of $E_6$ and
put $\sigma=\tau\circ\theta$, using the same symbols for the induced
actions on the split groups.
We write
\[
\begin{aligned}
 E&:=E'_{\mathrm{ad},\sigma}(E_6,R),&
 G&:=G_{\mathrm{ad},\sigma}(E_6,R),\\
 E_{\mathrm{sc}}&:=E'_{\mathrm{sc},\sigma}(E_6,R),&
 G_{\mathrm{sc}}&:=G_{\mathrm{sc},\sigma}(E_6,R).
\end{aligned}
\]
The full groups are the fixed-point groups under $\sigma$;
their elementary subgroups are generated by relative root elements.
The relative root system $\Phi_\rho$ has type $F_4$.
Long roots have parameters in $R_0$, and short roots have parameters
in $R$.
The coordinates are fixed in
Subsection~\ref{subsec:relative-root-system-and-blocks}.

The weight lattice modulo the root lattice of $E_6$ is cyclic of
order three, so these are the two standard isogeny types.
In the simply connected form,
\[
 G_{\mathrm{sc}}=E_{\mathrm{sc}},
\]
because the fixed torus is elementary and the local
torus--elementary factorization applies; see
\cite[Proposition~5.8]{GMNormalSubgroups} and
Subsection~\ref{subsec:simply-connected-form}.
For $\theta=\operatorname{id}_R$, the corresponding groups are the
split groups of type $F_4$.

Put
\[
 \operatorname{Aut}(R,\theta)
 :=\{f\in\operatorname{Aut}(R)\mid f\theta=\theta f\}.
\]
Each such $f$ induces a coefficientwise automorphism $f_*$.
For $g\in G$, write $i_g(x)=gxg^{-1}$; its restriction to $E$ is
defined because $E$ is normal in $G$.
Strictly inner and diagonal automorphisms are understood as in
\cite{BuninaMakadiya2A}.

\begin{maintheorem}
Let $(R,\theta)$ satisfy the preceding assumptions.
Every automorphism of $E$, and every automorphism of $G$, has a
unique expression
\[
 \varphi=i_g\circ f_*,
 \qquad
 g\in G,\quad f\in\operatorname{Aut}(R,\theta).
\]
Every automorphism of $G_{\mathrm{sc}}=E_{\mathrm{sc}}$ is a
composition of a strictly inner automorphism, a diagonal
automorphism, and a ring automorphism induced by an element of
$\operatorname{Aut}(R,\theta)$.
Conversely, each such composition is an automorphism of the
corresponding group.
\end{maintheorem}

There is no independent graph factor: on the $\sigma$-fixed groups,
\[
 \tau_*\theta_*(x)=x,\qquad \tau_*(x)=\theta_*(x),
\]
so the graph action is already included among the ring automorphisms.
No central factor is needed either: the adjoint groups are
centerless, and the simply connected full group is elementary and
perfect.
For an automorphism of $E$, the conjugating element is allowed to
lie in $G$, not necessarily in $E$.

\medskip
\noindent\textbf{Outline of the proof.}
We separate normalization of a finite family of constant elements
from recovery of the arbitrary root parameters.

\smallskip
\noindent\emph{1. Normalize the standard frame.}
For $\varphi\in\operatorname{Aut}E$, consider
\[
 \mathcal F(R)=(x_\lambda(1))_{\lambda\in\Phi_\rho},
\]
called the standard frame.
Normal structure allows reduction modulo $J$, and the field
automorphism theorem describes the residual action.
We lift and remove only its conjugating factor: the residual
coefficient automorphism already fixes the frame.
The resulting automorphism $\varphi_1$ therefore has standard
reductions on $\mathcal F(R)$.

On the $78$-dimensional Chevalley module $L_R$, one averaging over
a finite $2$-subgroup fixes the toral involutions and a chosen family
of Weyl elements simultaneously.
We then normalize the remaining Weyl elements and one complete
constant short-root matrix.
Weyl conjugation gives the other constant short-root elements, and
short--short commutators determine the constant long-root elements
from their squares.
The accumulated change of basis
$C_0\in\operatorname{GL}_{78}(R)$ is compatible with the twisting
and satisfies
\[
 C_0\equiv I_{78}\pmod{M_{78}(J)},\qquad
 C_0\varphi_1(x_\lambda(1))C_0^{-1}=x_\lambda(1)
 \quad(\lambda\in\Phi_\rho).
\]

\smallskip
\noindent\emph{2. Replace the linear normalization by a group element.}
At this point, $C_0$ is not known to belong to $G$.
Its transported Chevalley bracket is invariant under the standard
frame.
An integral calculation describes all frame-invariant alternating
tensors, and the Jacobi identity reduces the transported bracket
to two scalar coefficients.
We correct these coefficients over a faithfully flat local extension
without changing the normalized frame.
A local bracket-stabilizer result places the corrected matrix in
the adjoint group.
The triviality of the frame centralizer in the adjoint group over
arbitrary coefficient rings gives uniqueness and descent to $R$.

\smallskip
\noindent\emph{3. Recover the coefficients.}
The normalized map is now an automorphism of $E$ fixing the
standard frame. Its values preserve the Chevalley bracket.
Together with this bracket preservation, finite commutation
conditions recognize long-root elements and short-root elements
with skew parameters.
The $B_2$ relations and the fixed--skew decomposition of $R$ give
one additive coefficient map commuting with $\theta$.
A short $A_2$ relation proves multiplicativity; applying the same
argument to the inverse proves bijectivity.
Characteristicity and the trivial centralizer of $E$ give the full
adjoint result, while the central isogeny and perfectness give the
simply connected classification.

The remainder of Section~1 fixes the matrix conventions.
Sections~2--4 perform the linear normalization, and Section~5
corrects and descends the change of basis.
Section~6 recovers the coefficients and completes the classification.
It also gives coefficient recovery for frame-fixing homomorphisms
over arbitrary rings with involution, and relative normalization
for homomorphisms with a local target and prescribed standard
residual action.


\subsection{The relative root system and standard blocks}
\label{subsec:relative-root-system-and-blocks}

We fix the root coordinates, Chevalley signs, and block conventions
used in the normalization.

Number the $E_6$ diagram along the chain
$\alpha_1,\alpha_2,\alpha_3,\alpha_5,\alpha_6$, with
$\alpha_4$ attached to $\alpha_3$.
The graph involution $\tau$ interchanges $\alpha_1,\alpha_6$
and $\alpha_2,\alpha_5$ and fixes $\alpha_3,\alpha_4$.
The simple roots of the folded system $\Phi_\rho$ are
\[
 [\alpha_1]=\{\alpha_1,\alpha_6\},\qquad
 [\alpha_2]=\{\alpha_2,\alpha_5\},\qquad
 [\alpha_3]=\{\alpha_3\},\qquad
 [\alpha_4]=\{\alpha_4\}.
\]
We identify $[\alpha]$ with $(\alpha+\tau\alpha)/2$ in the invariant
real root space.
Singleton classes are long roots; two-element classes are short
roots, and their two absolute roots are orthogonal.

\smallskip
\noindent\emph{The Chevalley module and conventions.}
Let $\Phi=\Phi(E_6)$ and put
\[
 L_{\mathbb Z}
   =\bigoplus_{i=1}^6\mathbb Z h_i
      \oplus\bigoplus_{\alpha\in\Phi}\mathbb Z e_\alpha,
 \qquad
 L_R=L_{\mathbb Z}\otimes_{\mathbb Z}R,
 \qquad
 H=\bigoplus_{i=1}^6Rh_i,
\]
where $h_i=h_{\alpha_i}$.
Choose a $\tau$-stable Chevalley basis in which the matrix $Q$ of
$\tau$ satisfies
\[
 Qh_i=h_{\tau(\alpha_i)},\qquad
 Qe_\alpha=e_{\tau\alpha},\qquad
 [e_\alpha,e_{-\alpha}]=h_\alpha.
\]
For $\alpha=\sum_i m_i\alpha_i$, we have
$h_\alpha=\sum_i m_i h_i$.
Write
\[
 [e_\alpha,e_\beta]
   =N_{\alpha,\beta}e_{\alpha+\beta}
 \qquad(\alpha,\beta,\alpha+\beta\in\Phi).
\]
The signs needed below are fixed by
\[
\begin{gathered}
 N_{\alpha_3,\alpha_4}
 =N_{\alpha_3,\alpha_2}
 =N_{\alpha_3,\alpha_5}
 =N_{\alpha_1,\alpha_2}
 =N_{\alpha_6,\alpha_5}=1,\\
 N_{\alpha_5,\alpha_2+\alpha_3}
 =N_{\alpha_2,\alpha_3+\alpha_5}=1.
\end{gathered}
\]
These choices are compatible with $\tau$.
The equality of the last two constants follows from Jacobi and
$[e_{\alpha_2},e_{\alpha_5}]=0$.
Complete these choices to a fixed $\tau$-stable Chevalley basis,
retaining the normalization of opposite root vectors. Then
\[
 N_{-\alpha,-\beta}=-N_{\alpha,\beta}.
\]
Other root-string signs are retained in the structure constants;
no global positivity convention is used.

All matrices act on column vectors.
For group commutators we use $[x,y]=xyx^{-1}y^{-1}$.
The twisting on matrices is
\[
 \sigma(M)=Q\overline M Q^{-1},\qquad Q^2=I_{78},
\]
where the bar is applied entrywise.
We call $M$ \emph{$Q$-symmetric} if $\sigma(M)=M$.
This is semilinear compatibility, not the relation $QM=MQ$.

Put $D_\alpha=\operatorname{ad}e_\alpha$.
The absolute root matrix is
\[
 x_\alpha(t)
   =I_{78}+tD_\alpha+\frac{t^2}{2}D_\alpha^2,
 \qquad D_\alpha^3=0.
\]
The notation $\exp(tD_\alpha)$ refers only to this finite polynomial.
Explicitly,
\begin{align*}
 x_\alpha(t)e_\beta
   &=
   \begin{cases}
    e_\beta+N_{\alpha,\beta}t\,e_{\alpha+\beta},
       &\alpha+\beta\in\Phi,\\
    e_\beta,
       &\alpha+\beta\notin\Phi,\quad\beta\ne-\alpha,\\
    e_{-\alpha}+t h_\alpha-t^2e_\alpha,
       &\beta=-\alpha,
   \end{cases}\\
 x_\alpha(t)h
   &=h-t\alpha(h)e_\alpha
   \qquad(h\in H).
\end{align*}
These formulas and the Chevalley and Weyl relations determine the
matrices used below; see
\cite{SteinbergLectures,PlotkinVavilov,AbeTwisted,GMNormalSubgroups}.
The archive \texttt{twisted\_e6\_verification.zip} specifies the
complete basis, sign choices, and matrix constructions.
The accompanying \texttt{verification\_report.pdf} summarizes the
exact certificates used in the proof.

\subsubsection{Relative root elements}

For a long class $\lambda=\{\alpha\}$, abbreviate
$e_\lambda=e_\alpha$ and $h_\lambda=h_\alpha$.
For each short class $\mu=[\beta]$, choose a representative $\beta$
and use $-\beta$ for the opposite class.
The relative root elements are
\[
\begin{aligned}
 x_\lambda(s)&=x_\alpha(s),
   &&\lambda=\{\alpha\},\quad s\in R_0,\\
 x_\mu(t)&=x_\beta(t)x_{\tau\beta}(\overline t),
   &&\mu=\{\beta,\tau\beta\},\quad t\in R.
\end{aligned}
\]
The short-root factors commute by orthogonality, so
\[
 x_\mu(t)x_\mu(u)=x_\mu(t+u).
\]
Replacing $\beta$ by $\tau\beta$ replaces $t$ by $\overline t$.

For either root length, put
\[
 V_{\pm\lambda}
   =\bigoplus_{\alpha\in\lambda}
      (Re_\alpha\oplus Re_{-\alpha}),
 \qquad
 K_\lambda=H\oplus V_{\pm\lambda}.
\]
For a positive system $\Phi_\rho^+$,
\[
 L_R=H\oplus
       \bigoplus_{\lambda\in\Phi_\rho^+}V_{\pm\lambda}.
\]
There are twelve long-root pair spaces of rank two and twelve
short-root pair spaces of rank four.
Weyl elements permute these spaces, whereas root elements may join
several of them into larger invariant blocks.

Each relative root element preserves the corresponding
$K_\lambda$, where its action follows from the absolute formulas above.
The remaining blocks come from root strings.
Every block displayed below is a restriction to the stated
submodule; its complement is not assumed to be fixed.

\subsubsection{Standard root-string blocks}

We use one absolute string template for the long-root blocks and
the short $A_2$ blocks, followed by a separate mixed $B_2$ block.

Let $\beta,\beta+\alpha$ be an absolute two-root string and put
$\varepsilon=N_{\alpha,\beta}\in\{1,-1\}$.
In the signed basis
\[
 \mathcal B_{\alpha,\beta}
   :=\bigl(
       \varepsilon e_{\beta+\alpha},\ e_\beta,\
       -e_{-\beta},\ \varepsilon e_{-(\beta+\alpha)}
     \bigr),
\]
the block of $x_\alpha(s)$ is
\[
 M_{A_2}(s):=
 \begin{pmatrix}
 1&s&0&0\\
 0&1&0&0\\
 0&0&1&s\\
 0&0&0&1
 \end{pmatrix},
\]
since $N_{\alpha,-(\beta+\alpha)}=-\varepsilon$.

For the mixed subsystem generated by $[\alpha_3]$ and $[\alpha_2]$,
and the short $A_2$ subsystem generated by $[\alpha_1]$ and
$[\alpha_2]$, respectively, use the concatenated bases
\[
\begin{aligned}
 \mathcal B_{B_2}^{\mathrm{long}}
   &:=\bigl(\mathcal B_{\alpha_3,\alpha_2},
            \mathcal B_{\alpha_3,\alpha_5}\bigr),\\
 \mathcal B_{A_2}^{\mathrm{short}}
   &:=\bigl(\mathcal B_{\alpha_1,\alpha_2},
            \mathcal B_{\alpha_6,\alpha_5}\bigr).
\end{aligned}
\]
The corresponding blocks of $x_{[\alpha_3]}(s)$ and
$x_{[\alpha_1]}(t)$ are
\[
\begin{aligned}
 M_{B_2}^{\mathrm{long}}(s)
   &=M_{A_2}(s)\oplus M_{A_2}(s),\\
 M_{A_2}^{\mathrm{short}}(t)
   &=M_{A_2}(t)\oplus M_{A_2}(\overline t).
\end{aligned}
\]

For the short-root element $x_{[\alpha_2]}(t)$ in the mixed
$B_2$ subsystem, the mixed block acts on a different submodule.
Put $\eta=\alpha_2+\alpha_3+\alpha_5$ and use
\[
\begin{aligned}
 \mathcal B_{B_2}:=\bigl(&
 -e_\eta,\ -e_{\alpha_3+\alpha_5},\
 -e_{\alpha_2+\alpha_3},\ e_{\alpha_3},\\
 &-e_{-\alpha_3},\ -e_{-(\alpha_2+\alpha_3)},\
 -e_{-(\alpha_3+\alpha_5)},\ e_{-\eta}
 \bigr).
\end{aligned}
\]
Then
\[
 M_{B_2}^{\mathrm{short}}(t)=B_4(t)\oplus B_4(t),
 \qquad
 B_4(t)=
 \begin{pmatrix}
 1&t&\overline t&t\overline t\\
 0&1&0&\overline t\\
 0&0&1&t\\
 0&0&0&1
 \end{pmatrix}.
\]
For example, the fourth and eighth columns give
\begin{align*}
 x_{[\alpha_2]}(t)e_{\alpha_3}
   &=e_{\alpha_3}
       -t e_{\alpha_2+\alpha_3}
       -\overline t\,e_{\alpha_3+\alpha_5}
       -t\overline t\,e_\eta,\\
 x_{[\alpha_2]}(t)e_{-\eta}
   &=e_{-\eta}
       -\overline t\,e_{-(\alpha_2+\alpha_3)}
       -t e_{-(\alpha_3+\alpha_5)}
       -t\overline t\,e_{-\alpha_3}.
\end{align*}
The mixed terms come from multiplying the two commuting absolute
root matrices; the negative-string signs follow from the Chevalley
relations.

Other configurations follow from the absolute action formulas or
signed Weyl transport.
The transported basis and the chosen orbit representatives determine
the signs and any interchange of $t,\overline t$.

\subsubsection{Weyl elements}

The standard Weyl words are
\[
 w_\gamma(1)
   =x_\gamma(1)x_{-\gamma}(-1)x_\gamma(1)
   =\prod_{\alpha\in\gamma}w_\alpha(1).
\]
The product has one or two commuting factors.
These matrices are integral and commute with $Q$.
They are signed monomial on the absolute root spaces and act by
products of reflections on $H$.

On its own absolute root pair, $w_\alpha(1)$ interchanges
$e_\alpha$ and $-e_{-\alpha}$.
Thus its block in the signed basis $(e_\alpha,-e_{-\alpha})$ is
\[
 P=\begin{pmatrix}0&1\\1&0\end{pmatrix}.
\]
Writing $S_\gamma=w_\gamma(1)|_H$, the block on $K_\gamma$ is
$S_\gamma\oplus P$ for a long root and
$S_\gamma\oplus P\oplus P$ for a short root, in the corresponding
signed pair bases.

On $\mathcal B_{\alpha,\beta}$, the block of $w_\alpha(1)$ is
\[
 J_4=
 \begin{pmatrix}
 0&1&0&0\\
 -1&0&0&0\\
 0&0&0&1\\
 0&0&-1&0
 \end{pmatrix}.
\]
Consequently, the long mixed block on
$\mathcal B_{B_2}^{\mathrm{long}}$ and the short $A_2$ block on
$\mathcal B_{A_2}^{\mathrm{short}}$ are both $J_4\oplus J_4$.
On $\mathcal B_{B_2}$, the block of $w_{[\alpha_2]}(1)$ is
\[
 \Omega_8=\Omega_4\oplus\Omega_4,
 \qquad
 \Omega_4=
 \begin{pmatrix}
 0&0&0&1\\
 0&0&-1&0\\
 0&-1&0&0\\
 1&0&0&0
 \end{pmatrix}.
\]

The Cartan blocks are determined by the reflection formula
\[
 s_{\alpha_i}(h_j)
   =h_j-\langle\alpha_i,\alpha_j^\vee\rangle h_i,
 \qquad
 s_{\alpha_i}=w_{\alpha_i}(1)|_H.
\]
In particular, for $S_i=w_{[\alpha_i]}(1)|_H$,
\[
 S_1=s_{\alpha_1}s_{\alpha_6},\qquad
 S_2=s_{\alpha_2}s_{\alpha_5},\qquad
 S_3=s_{\alpha_3},\qquad
 S_4=s_{\alpha_4}.
\]

A Weyl element may stabilize a short relative root while exchanging
its two absolute roots.
The following lemma controls these transitions and the use of
repeated blocks.

\begin{lemma}[Signed transition blocks]
\label{lemma:signed-transition-blocks}
Let $N_\rho$ be generated by the standard relative Weyl
representatives, and let $\pi:N_\rho\to W(F_4)$ be the natural
projection.
\begin{enumerate}
\item
For long roots $\lambda,\nu$ and
$\pi(n)\{\pm\lambda\}=\{\pm\nu\}$, the transition matrix in the
signed bases
\[
 \mathcal B_\lambda=(e_\lambda,-e_{-\lambda}),\qquad
 \mathcal B_\nu=(e_\nu,-e_{-\nu})
\]
belongs to $\{\pm I_2,\pm P\}$.

\item
For short roots $\mu=[\beta]$, $\nu=[\gamma]$ and
$\pi(n)\{\pm\mu\}=\{\pm\nu\}$, use
\[
\begin{aligned}
 \mathcal B_\mu
   &=(e_\beta,-e_{-\beta},
      e_{\tau\beta},-e_{-\tau\beta}),\\
 \mathcal B_\nu
   &=(e_\gamma,-e_{-\gamma},
      e_{\tau\gamma},-e_{-\tau\gamma}).
\end{aligned}
\]
With
\[
 P_4=P\oplus P,\qquad
 Q_4=\begin{pmatrix}0&I_2\\I_2&0\end{pmatrix},
\]
the transition matrix belongs to
\[
 \{\pm I_4,\pm P_4,\pm Q_4,\pm P_4Q_4\}.
\]

\item
An endomorphism with the same block $B\in M_2(R)$ on every
long-root pair space commutes with $N_\rho$ on their sum if
$BP=PB$.
The analogous assertion holds for a repeated short-root block
$A\in M_4(R)$ if
\[
 AP_4=P_4A,\qquad AQ_4=Q_4A.
\]
On the whole module, also require the Cartan block to commute
with each $S_i$.
\end{enumerate}
\end{lemma}

\begin{proof}
Work first in the integral Chevalley lattice.
A standard Weyl word is signed monomial on the root spaces.
If its absolute Weyl image is $w$ and
\[
 ne_\alpha=\varepsilon e_{w\alpha},\qquad
 ne_{-\alpha}=\varepsilon'e_{-w\alpha},
\]
then bracket preservation gives
\[
 h_{w\alpha}=nh_\alpha
   =[ne_\alpha,ne_{-\alpha}]
   =\varepsilon\varepsilon'h_{w\alpha}.
\]
Thus $\varepsilon'=\varepsilon$.
In the signed pair bases, the images
$\varepsilon e_\nu$ and $\varepsilon e_{-\nu}$ therefore give
$\varepsilon I_2$ and $-\varepsilon P$, respectively.

Every element of $N_\rho$ commutes with $Q$, so
$ne_{\tau\beta}=Qne_\beta$.
The short-root transition is consequently determined by the image
of $e_\beta$ and may additionally exchange the two absolute pairs.
This gives the four possibilities up to sign in \textup{(2)}.
The commutation conditions in \textup{(3)} make a repeated block
intertwine every listed transition.
All these identities remain valid after base change to $R$.
\end{proof}

\subsubsection{Diagonal elements}

For every relative root $\gamma$, put
\[
 h_\gamma(-1)
   =w_\gamma(1)^2
   =\prod_{\alpha\in\gamma}h_\alpha(-1).
\]
These elements fix $H$ pointwise and act on absolute root vectors by
\[
 h_\gamma(-1)e_\beta
   =(-1)^{\sum_{\alpha\in\gamma}
                    \langle\beta,\alpha^\vee\rangle}e_\beta.
\]
The sign is constant on each $V_{\pm\lambda}$, since it is unchanged
by $\beta\mapsto-\beta$ and $\beta\mapsto\tau\beta$.
For the displayed blocks,
\[
 P^2=I_2,\qquad J_4^2=-I_4,\qquad \Omega_8^2=I_8,
\]
in agreement with the parity formula.

All formulas in this subsection are integral or defined over
$\mathbb Z[1/2]$ and remain valid under equivariant base change.
None uses locality of the coefficient ring.
\bigskip

\section{Standard preparations over local rings}
\label{sec:standard-preparations}

We remove the residual conjugating factor of an automorphism
$\varphi\in\operatorname{Aut}E$.
The resulting automorphism fixes the standard frame modulo the
Jacobson radical, providing the starting point for Section~3.

Retain the notation of Section~1:
\[
 J=\operatorname{Jac}(R),\qquad k=R/J,\qquad 2\in R^\times.
\]
The involution preserves $J$ and induces one on $k$, again denoted
by a bar; put $k_0=\{a\in k\mid\overline a=a\}$.

We also record that $E$ is characteristic in $G$.
For $\theta\ne\operatorname{id}_R$, this follows from
\cite[Corollary~A.2]{GMNormalSubgroups}.
For $\theta=\operatorname{id}_R$, the groups are split of type
$F_4$ and $E=G$ over the local ring $R$; see \cite{AbeLocal}.
Thus restriction to $E$ will be available when we treat the full
group in Section~6.

\subsection{$Q$-symmetric basis changes}

Write
\[
 M_{78}(R)^\sigma
   =\{A\in M_{78}(R)\mid Q\overline A Q^{-1}=A\},
\]
and use analogous notation for
$\operatorname{GL}_{78}(R)^\sigma$ and
$\operatorname{SL}_{78}(R)^\sigma$.
The fixed matrix ring is an $R_0$-algebra, and
\[
 E\subseteq G\subseteq\operatorname{SL}_{78}(R)^\sigma.
\]
Conjugation by $C\in\operatorname{GL}_{78}(R)^\sigma$ preserves
this ring, since
\[
 \sigma(C^{-1}AC)=\sigma(C)^{-1}\sigma(A)\sigma(C).
\]
Such a matrix need not normalize $E$ or preserve the Chevalley
bracket.
Until the correction in Section~5, maps obtained by these basis
changes are therefore viewed as isomorphisms from $E$ onto
conjugate matrix groups.

\subsection{Reduction modulo the Jacobson radical}

We show that the congruence kernel is characteristic and that
the conjugating element in the residual classification lifts to $G$.
Put
\[
 E_k=E'_{\mathrm{ad},\sigma}(E_6,k),\qquad
 G_k=G_{\mathrm{ad},\sigma}(E_6,k),
\]
with the induced twisting.
Reduction gives a surjection $\pi_J:E\to E_k$:
short-root parameters lift arbitrarily, and a fixed residue
parameter with lift $a$ has the fixed lift $(a+\overline a)/2$.
Its kernel is
\[
 N_J=E\cap\bigl(I_{78}+M_{78}(J)\bigr).
\]

We claim that $N_J$ contains every proper normal subgroup of $E$.
For a proper $K\lhd E$, the normal-structure theorem supplies a
proper $\theta$-invariant ideal $I$ such that the image of $K$
in the elementary group over $R/I$ is central; see
\cite{SuzukiLocalTwisted} and
\cite[Theorem~1.4]{GMNormalSubgroups}.
For $\theta=\operatorname{id}_R$, use the split $F_4$
normal-structure theorem recalled in
\cite[Theorem~1.2]{GMNormalSubgroups}.
Since $R$ is local, $I\subseteq J$.
Reduction of the elementary group over $R/I$ onto $E_k$ is
surjective by the same parameter-lifting argument. Hence
\[
 \pi_J(K)\subseteq Z(E_k)=1.
\]
Here $E_k$ is the centerless adjoint elementary group of type
${}^2E_6$ or $F_4$, according to whether the residue involution
is nontrivial or trivial; see \cite{SteinbergLectures}.
Thus $N_J$ is the unique maximal proper normal subgroup of $E$,
and hence is characteristic. Consequently,
\[
 E/N_J\simeq E_k,
\]
and $\varphi$ induces an automorphism $\overline\varphi$ of $E_k$.

\begin{lemma}[Lifting the residual conjugator]
\label{lemma:surjective-reduction-G}
The reduction homomorphism $G\to G_k$ is surjective.
\end{lemma}

\begin{proof}
The Bruhat decomposition expresses every element of $G_k$ using
relative unipotent elements, standard Weyl representatives, and
a point of the fixed torus; see \cite{SteinbergLectures,Borel}.
The unipotent parameters lift as above, and the Weyl representatives
are constant root words.

For $U=R,k$, use the simple-root coordinates of the adjoint torus
\[
 \mathbf T_{\mathrm{ad}}(U)
   =\operatorname{Hom}(\mathbb Z\Phi,U^\times).
\]
In the ordering $\alpha_1,\ldots,\alpha_6$, its $\sigma$-fixed
points have coordinates
\[
 (a,b,c,d,\overline b,\overline a),
 \qquad a,b\in U^\times,\quad c,d\in U_0^\times,
\]
where $U_0$ is the fixed subring.
Lift $a,b$ to units of $R$ and lift $c,d$ by averaging.
The averaged lifts are still units, since their residues are
nonzero. These coordinates give a fixed torus lift.
\end{proof}

\begin{proposition}[Residual normalization]
\label{prop:reduction}
There exist $\widetilde u\in G$ and
$f\in\operatorname{Aut}(k,\overline{\ })$ such that the automorphism
\[
 \varphi_1:=i_{\widetilde u^{-1}}\circ\varphi
\]
satisfies
\[
 \pi_J\circ\varphi_1=f_*\circ\pi_J.
\]
In particular,
\[
 \varphi_1(a)\equiv a\pmod{M_{78}(J)}
 \qquad
 \bigl(a\in\langle\mathcal F(R)\rangle\bigr).
\]
This applies to every $x_\lambda(1)$, $w_\lambda(1)$, and
$h_\lambda(-1)$.
\end{proposition}

\begin{proof}
Suppose first that the residue involution is nontrivial.
The twisted field theorem
\cite[Theorem~36]{SteinbergLectures} gives inner, diagonal, and
field factors for $\overline\varphi$.
The compatible diagonal factors are represented by
$\mathbf T_{\mathrm{ad}}(k)^\sigma$:
their characters satisfy
\[
 \chi(\tau\alpha)=\overline{\chi(\alpha)}.
\]
They combine with the inner factor into conjugation by an element
$u\in G_k$.

The field factor commutes with the involution. Indeed, on a
short-root element it acts as
\[
 f_*\bigl(x_\alpha(t)x_{\tau\alpha}(\overline t)\bigr)
   =x_\alpha(f(t))x_{\tau\alpha}(f(\overline t)).
\]
Preservation of the twisted short-root subgroup and uniqueness
of the absolute root coordinates give
$f(\overline t)=\overline{f(t)}$. Thus
\[
 \overline\varphi=i_u\circ f_*,
 \qquad f\in\operatorname{Aut}(k,\overline{\ }).
\]

If the residue involution is trivial, then $E_k=G_k$ is split
of type $F_4$.
Its root and weight lattices coincide, so diagonal factors are
inner.
The field automorphism theorem
\cite[Theorem~30]{SteinbergLectures} for finite $k$, and
\cite{Humphreys1969} for infinite $k$, gives the same expression.
There is no additional graph factor in characteristic different
from two, and equivariance of $f$ is automatic.

Lift $u$ by Lemma~\ref{lemma:surjective-reduction-G} and remove
its conjugation. This gives the required residual identity.
Since $f$ fixes the prime subring, it fixes the reduction of
every element of the standard frame, and hence of every word
in that frame.
\end{proof}

Only the residual conjugator has been lifted; no lift of $f$ to $R$
is assumed.
Starting from $\varphi_1$, the first averaging in Section~3 will
fix the toral involutions and an initial family of Weyl
representatives simultaneously.

\bigskip


\section{Normalization of the Weyl elements}
\label{sec:weyl-images}

Starting with $\varphi_1$ from Proposition~\ref{prop:reduction},
we fix all relative Weyl representatives and toral involutions
by a $Q$-symmetric basis change congruent to the identity.
One averaging first fixes the toral involutions and an initial
family of Weyl elements simultaneously.
The diagonal eigenspaces then separate the Cartan, short-root,
and long-root parts; stabilizer actions and two small matrix
calculations complete the normalization.

Initially write $\varphi=\varphi_1$ and put
\[
 w_\lambda:=w_\lambda(1),\qquad
 W_\lambda:=\varphi(w_\lambda),\qquad
 \mathcal D:=\langle h_\lambda(-1)\mid\lambda\in\Phi_\rho\rangle.
\]
We retain these symbols for the current map and its images after
each basis change. By Proposition~\ref{prop:reduction},
\[
 \varphi(n)\equiv n\pmod{M_{78}(J)}
 \qquad(n\in N_\rho).
\]

\subsection{Initial averaging}
\label{subsec:initial-weyl-family}
\label{subsec:coordinate-short-weyl}

We choose a finite $2$-subgroup containing $\mathcal D$ and
normalize its entire action at once.

For the natural projection $\pi:N_\rho\to W(F_4)$, the Tits
relations give
\[
 \ker\pi=\mathcal D
   =\langle h_{[\alpha_i]}(-1)\mid1\leq i\leq4\rangle;
\]
see \cite{TitsWeyl,SteinbergLectures}.
Every orbit sum of absolute coroots is an integral combination
of the four simple orbit sums, so the four displayed involutions
generate all $h_\lambda(-1)$.
Thus $\mathcal D$ is a finite $2$-group, as is the inverse image
of every $2$-subgroup of $W(F_4)$.

The finite group $N_\rho$ has the same integral matrix model over
every nonzero ring in which $2$ is invertible.
Its signed monomial action on the absolute root vectors is faithful:
the identities $[e_{\alpha_i},e_{-\alpha_i}]=h_i$ also determine
the Cartan action.
Distinct signed monomial matrices remain distinct after such
base change.

\begin{lemma}[Finite averaging]
\label{lemma:finite-weyl-averaging}
Let $\mathcal P\leq N_\rho$ be finite with
$|\mathcal P|\in R^\times$.
For any homomorphism
$\varphi:\mathcal P\to\operatorname{GL}_{78}(R)^\sigma$
satisfying $\varphi(p)\equiv p\pmod{M_{78}(J)}$, the matrix
\[
 D_{\mathcal P}
   :=\frac1{|\mathcal P|}
        \sum_{p\in\mathcal P}\varphi(p)p^{-1}
\]
is $Q$-symmetric, invertible, and congruent to $I_{78}$ modulo
$M_{78}(J)$, and
\[
 D_{\mathcal P}^{-1}\varphi(p)D_{\mathcal P}=p
 \qquad(p\in\mathcal P).
\]
\end{lemma}

\begin{proof}
Each summand is $Q$-symmetric and congruent to the identity.
Since $|\mathcal P|^{-1}\in R_0$, the same holds for
$D_{\mathcal P}$, which is therefore invertible.
For $a\in\mathcal P$, reindexing by $r=ap$ gives
\[
 \varphi(a)D_{\mathcal P}
   =\frac1{|\mathcal P|}
      \sum_{p\in\mathcal P}\varphi(ap)p^{-1}
   =\frac1{|\mathcal P|}
      \sum_{r\in\mathcal P}\varphi(r)r^{-1}a
   =D_{\mathcal P}a.
\]
\end{proof}

Use the coordinate realization of $F_4$ with simple roots
\[
\begin{aligned}
 \beta_0:=\tfrac12(e_1-e_2-e_3-e_4)&=[\alpha_1],
 &e_4&=[\alpha_2],\\
 e_3-e_4&=[\alpha_3],
 &e_2-e_3&=[\alpha_4].
\end{aligned}
\]
Let $W_B=W(B_4)$ be the signed-permutation subgroup and put
$N_B=\pi^{-1}(W_B)$.
Let $W_0$ be the subgroup of $W_B$ preserving the partition
$\{\{1,2\},\{3,4\}\}$.
Its elements may change coordinate signs, permute coordinates
within either pair, and interchange the pairs. Hence
\[
 |W_0|=2^4\cdot8=2^7,\qquad
 \mathcal P:=\pi^{-1}(W_0)
\]
is a finite $2$-group containing $\mathcal D$.

Apply Lemma~\ref{lemma:finite-weyl-averaging} and replace
$\varphi$ by $i_{D_{\mathcal P}^{-1}}\circ\varphi$.
Then
\[
 \varphi(p)=p\qquad(p\in\mathcal P).
\]
In particular, $\mathcal D$ and the following elements are fixed:
\[
\begin{gathered}
 w_{e_i}\quad(1\leq i\leq4),\qquad
 w_{e_1\pm e_2},\qquad w_{e_3\pm e_4},\\
 g:=w_{e_2-e_3}w_{e_3-e_4}w_{e_1-e_2}w_{e_2-e_3}.
\end{gathered}
\]
Here $\pi(g)=(1\,3)(2\,4)$.
The word $g$ is fixed because it belongs to $\mathcal P$;
its individual factors need not yet be fixed.
All subsequent basis changes will commute with $\mathcal D$.

\subsection{The diagonal decomposition}
\label{subsec:weyl-diagonal-decompositions}

The fixed diagonal family gives
\[
 W_\beta h_\mu(-1)W_\beta^{-1}
   =h_{s_\beta(\mu)}(-1).
\]
Thus $W_\beta$ permutes the $\mathcal D$-eigenspaces in the same
way as $w_\beta$.

Let $\mathbf c(\lambda)\in(\mathbb Z/2\mathbb Z)^4$ record the
signs of the four simple diagonal involutions, in the order
$[\alpha_1],\ldots,[\alpha_4]$:
\[
 h_\mu(-1)|_{V_{\pm\lambda}}
   =(-1)^{\langle\lambda,\mu^\vee\rangle}I,
 \qquad h_\mu(-1)|_H=I.
\]
For $i<j$, put
\[
\begin{gathered}
 L_{ij}:=V_{\pm(e_i-e_j)}\oplus V_{\pm(e_i+e_j)},\\
 U_0:=L_{12}\oplus L_{34},\qquad
 U_1:=L_{13}\oplus L_{24},\qquad
 U_2:=L_{14}\oplus L_{23}.
\end{gathered}
\]
The long-root sector has three characters:
\[
\begin{array}{c|ccc}
 \text{space}&U_0&U_1&U_2\\ \hline
 \mathbf c&(0,0,0,1)&(0,0,1,1)&(0,0,1,0).
\end{array}
\]
For short roots, write
\[
\begin{gathered}
 E_i:=V_{\pm e_i},\qquad
 \delta_{\varepsilon_2\varepsilon_3\varepsilon_4}
   :=\tfrac12(e_1+\varepsilon_2e_2+
                    \varepsilon_3e_3+\varepsilon_4e_4),\\
 T_{\varepsilon_2\varepsilon_3\varepsilon_4}
   :=V_{\pm\delta_{\varepsilon_2\varepsilon_3\varepsilon_4}}.
\end{gathered}
\]
Their twelve pair spaces have distinct characters:
\[
\begin{array}{c|c@{\qquad}c|c}
 E_1&(1,0,0,0)&E_2&(1,0,0,1)\\
 E_3&(1,0,1,1)&E_4&(1,0,1,0)\\ \hline
 T_{+++}&(1,1,0,0)&T_{++-}&(0,1,1,0)\\
 T_{+-+}&(0,1,1,1)&T_{+--}&(1,1,0,1)\\
 T_{-++}&(0,1,0,1)&T_{-+-}&(1,1,1,1)\\
 T_{--+}&(1,1,1,0)&T_{---}&(0,1,0,0).
\end{array}
\]
These tables follow from the parity formulas of Section~1.
Thus $H$ is the trivial eigenspace, each short-root pair is an
individual eigenspace, and the long-root eigenspaces are
$U_0,U_1,U_2$.
Set
\[
 L_\ell:=U_0\oplus U_1\oplus U_2,\qquad
 L_s:=\bigoplus_{\mu\in\Phi_{\rho,s}/\{\pm1\}}V_{\pm\mu}.
\]
Both $N_\rho$ and $\varphi(N_\rho)$ preserve
\[
 L_R=H\oplus L_\ell\oplus L_s.
\]
All projections associated with these decompositions are
$Q$-symmetric and commute with $\mathcal D$.

\subsection{Transport between blocks}
\label{subsec:weyl-normalization-lemmas}

We use the stabilizer of one block to normalize the action on
its entire orbit.
Since $\varphi$ fixes $\mathcal D$, for $n\in N_\rho$ and
$d\in\mathcal D$ we have
\[
 d\,\varphi(n)n^{-1}
   =\varphi(n)\varphi(n^{-1}dn)n^{-1}
   =\varphi(n)n^{-1}d.
\]
Consequently, $\varphi(n)n^{-1}$ is $Q$-symmetric, congruent to
the identity, and commutes with $\mathcal D$.

\begin{lemma}[Transport normalization]
\label{lemma:weyl-transport}
Let $\mathcal N\leq N_\rho$ and suppose that
\[
 L_R=L'\oplus L,\qquad
 L=\bigoplus_{s\in\mathcal S}V_s,
\]
where $\mathcal S$ is finite.
Assume that all associated projections are $Q$-symmetric and
commute with $\mathcal D$, that $\mathcal N$ and
$\varphi(\mathcal N)$ preserve $L'$ and $L$, and that
$\mathcal N$ permutes the $V_s$ transitively.
Fix $V_0=V_{s_0}$, with stabilizer $\mathcal N_0$.
If
\[
 \varphi(a)|_{V_0}=a|_{V_0}\qquad(a\in\mathcal N_0),
\]
then there is a $Q$-symmetric
$C\equiv I_{78}\pmod{M_{78}(J)}$, equal to the identity on $L'$
and commuting with $\mathcal D$, such that
\[
 C^{-1}\varphi(n)C|_L=n|_L\qquad(n\in\mathcal N).
\]
Moreover, $C$ commutes with every $n\in\mathcal N$ for which
$\varphi(n)=n$.
\end{lemma}

\begin{proof}
Choose $n_s\in\mathcal N$ with $n_sV_0=V_s$ and $n_{s_0}=1$.
If $p'$ and $p_s$ are the projections onto $L'$ and $V_s$, set
\[
 C=p'+\sum_{s\in\mathcal S}\varphi(n_s)n_s^{-1}p_s.
\]
The preceding observations show that $C$ is $Q$-symmetric,
commutes with $\mathcal D$, and satisfies
\[
 C|_{L'}=I,\qquad C(L)\subseteq L,\qquad
 C\equiv p'+\sum_s p_s=I_{78}\pmod{M_{78}(J)}.
\]
In particular, $C$ is invertible.

If $nV_s=V_t$, then $a=n_t^{-1}nn_s\in\mathcal N_0$.
For $v\in V_s$, the vector $n_s^{-1}v$ lies in $V_0$, so
\[
\begin{aligned}
 \varphi(n)Cv
   &=\varphi(n_t)\varphi(a)n_s^{-1}v\\
   &=\varphi(n_t)a n_s^{-1}v\\
   &=\varphi(n_t)n_t^{-1}nv=Cnv.
\end{aligned}
\]
This proves the intertwining on $L$.
When $\varphi(n)=n$, it also proves $nC=Cn$ there.
The same holds on $L'$, since $n$ preserves $L'$ and
$C|_{L'}=I$.
\end{proof}

The lemma does not assume that $\varphi(\mathcal N)$ preserves
or permutes the individual summands $V_s$.

We also use repeatedly the implication
\[
 A\equiv I_n\pmod{M_n(J)},\qquad A^2=I_n
 \quad\Longrightarrow\quad A=I_n.
\]
Indeed, $A+I_n\equiv2I_n\pmod{M_n(J)}$ is invertible and
$(A-I_n)(A+I_n)=0$.

\subsection{The Cartan part}
\label{subsec:weyl-cartan-part}

We first determine the action of $N_B$ on $H$, then normalize
$w_{\beta_0}$ on its two-dimensional $Q$-negative part.

Use the basis
\[
\begin{aligned}
 q_1&=\tfrac12(h_{e_1-e_2}+h_{e_1+e_2}),&
 q_2&=\tfrac12(h_{e_1+e_2}-h_{e_1-e_2}),\\
 q_3&=\tfrac12(h_{e_3-e_4}+h_{e_3+e_4}),&
 q_4&=\tfrac12(h_{e_3+e_4}-h_{e_3-e_4}),\\
 z_+&=2(h_1-h_6)+(h_2-h_5),&
 z_-&=h_2-h_5.
\end{aligned}
\]
Its determinant in the Chevalley basis is $4$. Thus
\[
 H=H^+\oplus H^-,\qquad
 H^+=\langle q_1,q_2,q_3,q_4\rangle_R,\qquad
 H^-=\langle z_+,z_-\rangle_R,
\]
where $H^\pm=H\cap\ker(Q\mp I)$.
The fixed element $w_{e_i}$ changes the signs of $q_i,z_-$
and fixes the other basis vectors.
The six resulting simultaneous characters are distinct.

Put $T=W_{e_2-e_3}|_H$ and $S=w_{e_2-e_3}|_H$.
On $H$, conjugation by $T$ interchanges the second and third
coordinate reflections and fixes the other two.
Since $T^2=I$ and $T\equiv S\pmod J$, the character decomposition
gives
\[
 Tq_2=pq_3,\qquad Tq_3=p^{-1}q_2,\qquad
 Tq_1=q_1,\quad Tq_4=q_4,\quad Tz_\pm=z_\pm
\]
for some $p\in1+J$.
The Weyl image of $w_{e_2-e_3}gw_{e_2-e_3}^{-1}$ is
$(1\,2)(3\,4)$, so this element belongs to the fixed group
$\mathcal P$.
Writing $G_H=g|_H$, we obtain
\[
 TG_HT=SG_HS.
\]
The two sides send $q_1$ to $p^{-1}q_2$ and $q_2$, respectively.
Hence $p=1$ and $T=S$.
Since $W_0$ and $s_{e_2-e_3}$ generate $W_B$, every element of
$N_B$ now acts standardly on $H$.

Next put $T=W_{\beta_0}|_H$ and $b_H=w_{\beta_0}|_H$.
The long-root Weyl group $W(D_4)$ is normal in $W(F_4)$,
and its inverse image in $N_\rho$ lies in $N_B$.
Since these representatives already act standardly on $H$,
the matrix $b_H^{-1}T$ centralizes their action.
On $H^+$ this is the even signed-permutation action in the
$q_i$-basis; on $H^-$ it is trivial.
The even sign changes separate the four $q_i$-lines from each
other and from $H^-$, and coordinate permutations make their
four scalars equal.
Thus the centralizer consists of
\[
 aI_{H^+}\oplus A,\qquad a\in R,\quad A\in M_2(R).
\]
In particular, $T|_{H^+}=a b_H|_{H^+}$ for $a\in1+J$.
The equality $T^2=I$ gives $a^2=1$, hence $a=1$.

On $H^-$, in the basis $(z_+,z_-)$, the standard restrictions are
\[
 b_H|_{H^-}
   =\frac12\begin{pmatrix}-1&1\\3&1\end{pmatrix},
 \qquad
 S_-:=w_{e_4}|_{H^-}
   =\begin{pmatrix}1&0\\0&-1\end{pmatrix}.
\]
Write
\[
 T_-:=T|_{H^-}=\begin{pmatrix}r&s\\t&u\end{pmatrix}.
\]
Its entries lie in $R_0$ by $Q$-symmetry, and
$s\in\tfrac12+J$ is a unit.
The $(1,2)$-entries of the square and simple-root braid relations
\[
 T_-^2=I,\qquad T_-S_-T_-=S_-T_-S_-
\]
give $r+u=0$ and $r-u+1=0$.
Hence $r=-\tfrac12$, $u=\tfrac12$, and $st=\tfrac34$, so
\[
 T_-=\frac12
 \begin{pmatrix}
 -1&\zeta\\
 3\zeta^{-1}&1
 \end{pmatrix},
 \qquad \zeta=2s\in(1+J)\cap R_0.
\]
Let $K_H$ multiply $z_-$ by $\zeta^{-1}$ and fix the other
five Cartan basis vectors. Then
\[
 K_H^{-1}TK_H=b_H.
\]
Every element of $N_B$ acts on $H^-$ as $I$ or $S_-$, so
$K_H$ commutes with its action on $H$.
Extend $K_H$ by the identity on $L_\ell\oplus L_s$.
This $Q$-symmetric matrix is congruent to the identity.
Replacing $\varphi$ by $i_{K_H^{-1}}\circ\varphi$ preserves
the fixed group $\mathcal P$ and makes $w_{\beta_0}$ standard
on $H$.
Since $N_B$ and $w_{\beta_0}$ generate $N_\rho$, we now have
\[
 \varphi(n)|_H=n|_H\qquad(n\in N_\rho).
\]
No cancellation of $3$ was used.

\subsection{The short-root sector}
\label{subsec:short-root-weyl-part}
\label{subsubsec:coordinate-short-blocks}
\label{subsec:transport-short-weyl}

It suffices to determine the stabilizer action on
$V_0=E_1=V_{\pm e_1}$ and then apply transport.
Its diagonal character is fixed by $s_{e_2-e_3}$, so
$W_{e_2-e_3}$ preserves $V_0$.
The standard root-string action gives
\[
 w_{e_2-e_3}|_{V_0}=I,\qquad
 h_{e_2-e_3}(-1)|_{V_0}=I.
\]
Thus $W_{e_2-e_3}|_{V_0}$ is an involution congruent to the
identity, and also equals $I$.

The stabilizer of $\{\pm e_1\}$ in $W(F_4)$ is generated by
\[
 s_{e_1},\qquad s_{e_2-e_3},\qquad
 s_{e_3-e_4},\qquad s_{e_4};
\]
the last three generate the $B_3$-subsystem orthogonal to $e_1$.
Its inverse image $N_0$ is generated by the corresponding
representatives and $\mathcal D$.
Apart from $w_{e_2-e_3}$, whose restriction was just determined,
all these generators belong to $\mathcal P$. Hence
\[
 \varphi(a)|_{V_0}=a|_{V_0}\qquad(a\in N_0).
\]

Apply Lemma~\ref{lemma:weyl-transport} with $\mathcal N=N_\rho$,
$L=L_s$, and the twelve short-root pair spaces.
The resulting change $C_s$ is the identity on $H\oplus L_\ell$
and commutes with $\mathcal P$.
After replacing $\varphi$ by $i_{C_s^{-1}}\circ\varphi$,
\[
 \varphi(n)|_{H\oplus L_s}=n|_{H\oplus L_s}
 \qquad(n\in N_\rho).
\]

\subsection{The long-root sector}
\label{subsec:long-root-weyl-part}

Only the action on $L_\ell$ remains.
We first normalize $N_B$ by transport between $U_0,U_1,U_2$,
then treat the additional generator $w_{\beta_0}$.

\subsubsection{The signed-permutation subgroup}
\label{subsubsec:long-weyl-transport}

The group $N_B$ acts transitively on $U_0,U_1,U_2$.
The stabilizer of $U_0$ is precisely $\mathcal P$, since its Weyl
image must preserve the partition $\{\{1,2\},\{3,4\}\}$.
Apply Lemma~\ref{lemma:weyl-transport} with $\mathcal N=N_B$,
$L=L_\ell$, and these three summands, using
\[
 n_0=1,\qquad n_1=w_{e_2-e_3},\qquad
 n_2=w_{e_3-e_4}w_{e_2-e_3},\qquad n_iU_0=U_i.
\]
The change is the identity on $H\oplus L_s$ and preserves
$\mathcal P$. Consequently,
\[
 \varphi(n)=n\qquad(n\in N_B)
\]
on the whole module.

\subsubsection{The half-sum representative}
\label{subsubsec:half-sum-weyl}

We determine $\widehat w:=W_{\beta_0}$ on $U_0$ and transport
the correction to $U_1,U_2$.
Put $w=w_{\beta_0}$.
The reflection $s_{\beta_0}$ fixes the three long-root characters,
so $\widehat w$ preserves each $U_i$.
Moreover, $\widehat w^2=I$ on $L_\ell$, because
$h_{\beta_0}(-1)|_{L_\ell}=I$.

Set
\[
\begin{gathered}
 \lambda_1=e_1-e_2,\quad \lambda_2=e_1+e_2,\quad
 \lambda_3=e_3-e_4,\quad \lambda_4=e_3+e_4,\\
 b_0=w|_{U_0},\qquad C_i=w_{e_i}|_{U_0},\qquad
 R_i=w_{\lambda_i}|_{U_0},\qquad G=g|_{U_0}.
\end{gathered}
\]
Define
\[
 u_1=e_{\lambda_1}-e_{-\lambda_1},\qquad
 v_1=e_{\lambda_1}+e_{-\lambda_1},
\]
and
\[
\begin{aligned}
 u_2&=C_2u_1,&u_3&=Gu_1,&u_4&=b_0u_1,\\
 v_2&=C_2v_1,&v_3&=Gv_1,&v_4&=b_0v_1.
\end{aligned}
\]
Each pair $(u_i,v_i)$ has change-of-basis determinant $\pm2$
in its root pair. Thus
\[
 \mathcal B_0=(u_1,u_2,u_3,u_4;v_1,v_2,v_3,v_4)
\]
is a basis of $U_0$.
The root-string formulas give the following columnwise action:
\[
\begin{array}{c|rrrr|rrrr}
 &u_1&u_2&u_3&u_4&v_1&v_2&v_3&v_4\\ \hline
 C_1&u_2&u_1&u_3&u_4&-v_2&-v_1&v_3&v_4\\
 C_2&u_2&u_1&u_3&u_4& v_2& v_1&v_3&v_4\\
 C_3&u_1&u_2&u_4&u_3&v_1&v_2&-v_4&-v_3\\
 C_4&u_1&u_2&u_4&u_3&v_1&v_2& v_4& v_3\\
 b_0&u_4&u_2&u_3&u_1&v_4&v_2&v_3&v_1\\
 G&u_3&u_4&u_1&u_2&v_3&v_4&v_1&v_2.
\end{array}
\]
The matrix $R_i$ changes the sign of $v_i$ and fixes the other
seven basis vectors.
Put
\[
 U_0^+=\langle u_1,u_2,u_3,u_4\rangle_R,\qquad
 U_0^-=\langle v_1,v_2,v_3,v_4\rangle_R.
\]
These superscripts do not refer to $Q$, which is the identity
on all of $U_0$.

Let $T=\widehat w|_{U_0}$.
The relations
\[
 TR_1=R_4T,\qquad TR_4=R_1T,\qquad
 TR_2=R_2T,\qquad TR_3=R_3T
\]
separate $U_0^+$ from the four $v$-lines.
Together with $T^2=I$ and $T\equiv b_0\pmod J$, they give
\[
\begin{gathered}
 T=T_+\oplus T_-,\\
 T_-v_1=pv_4,\qquad T_-v_4=p^{-1}v_1,\qquad
 T_-v_2=v_2,\qquad T_-v_3=v_3,\qquad p\in1+J.
\end{gathered}
\]
The Weyl image of $wgw^{-1}$ is
$\operatorname{diag}(1,-1,1,-1)$, so $wgw^{-1}\in N_B$ is fixed.
The action table therefore gives
\[
 TGT=b_0Gb_0=K,\qquad K:=C_2C_4.
\]
Applied to $v_1$, this identity gives $pv_2=v_2$.
Hence $p=1$ and $T_-=b_0|_{U_0^-}$.

On $U_0^+$ use the basis
\[
\begin{aligned}
 f_0&=u_1+u_2+u_3+u_4,&
 f_1&=u_1+u_2-u_3-u_4,\\
 f_2&=u_1-u_2+u_3-u_4,&
 f_3&=u_1-u_2-u_3+u_4.
\end{aligned}
\]
The inverse change has coefficients in $\tfrac14\mathbb Z$.
The commuting involutions $K,G$ have four distinct characters:
\[
\begin{array}{c|rrrr}
 &f_0&f_1&f_2&f_3\\ \hline
 K&1&1&-1&-1\\
 G&1&-1&1&-1.
\end{array}
\]
The standard matrix $b_0$ fixes $f_0,f_3$ and interchanges
$f_1,f_2$ with minus signs.
Since $TGT=K$ and $T^2=I$, also $TKT=G$.
Thus $T_+$ preserves $Rf_0,Rf_3$ and interchanges $Rf_1,Rf_2$.
Its square and residue give
\[
 T_+f_0=f_0,\qquad T_+f_3=f_3,\qquad
 T_+f_1=-\nu f_2,\qquad T_+f_2=-\nu^{-1}f_1
\]
for a unique $\nu\in1+J$.
Since $Q|_{U_0}=I$, $Q$-symmetry gives $\nu\in R_0$.

Let $Z_0(\nu)$ multiply $f_1$ by $\nu$ and fix
$f_0,f_2,f_3$ and $U_0^-$.
Then
\[
 T=Z_0(\nu)^{-1}b_0Z_0(\nu).
\]
The action table shows that $Z_0(\nu)$ commutes with every
$C_i,R_i,G$.
Their representatives, together with $\mathcal D$, generate
the stabilizer $\mathcal P$ of $U_0$ in $N_B$.
Since $\mathcal D$ acts by scalar signs on $U_0$, $Z_0(\nu)$
commutes with the entire stabilizer action.

Using the representatives $n_i$ chosen above, define
\[
 Z_\ell(\nu)|_{U_i}=n_iZ_0(\nu)n_i^{-1},\qquad
 Z_\ell(\nu)|_{H\oplus L_s}=I.
\]
If $nU_i=U_j$ for $n\in N_B$, then
$n_j^{-1}nn_i\in\mathcal P$.
Commutation with this stabilizer element gives
\[
 nZ_\ell(\nu)=Z_\ell(\nu)n.
\]
The matrix $Z_\ell(\nu)$ is $Q$-symmetric and congruent to the
identity.

Finally, $n_1,n_2$ commute with $w$: their defining long roots
are orthogonal to $\beta_0$.
Since these representatives are fixed, they also commute with
$\widehat w$.
Transporting the identity on $U_0$ gives
\[
 Z_\ell(\nu)\widehat w Z_\ell(\nu)^{-1}=w.
\]
Replacing $\varphi$ by $i_{Z_\ell(\nu)}\circ\varphi$ preserves
all of $N_B$ and the standard action on $H\oplus L_s$.
As $N_B$ and $w_{\beta_0}$ generate $N_\rho$, the normalization
is complete.

\begin{proposition}[Weyl normalization]
\label{prop:weyl-normalization}
After a $Q$-symmetric basis change congruent to $I_{78}$ modulo
$M_{78}(J)$, all relative Weyl representatives and diagonal
involutions are fixed:
\[
 \varphi(w_\lambda(1))=w_\lambda(1),\qquad
 \varphi(h_\lambda(-1))=h_\lambda(-1)
 \qquad(\lambda\in\Phi_\rho).
\]
\end{proposition}

The construction applies to any homomorphism
$N_\rho\to\operatorname{GL}_{78}(R)^\sigma$
whose reduction is the standard representation:
only the group relations and matrix congruences were used.

\bigskip


\section{Normalization of the constant root elements}
\label{sec:root-unipotents}

We now fix the constant root elements without changing the Weyl
normalization. Seven linear equations determine a change of basis
for one complete short-root matrix; an invertible Jacobian minor
then proves the full conjugacy. Weyl conjugation gives the other
constant short-root elements, and short--short commutators determine
the constant long-root elements from their squares.

Retain $\varphi$ from Proposition~\ref{prop:weyl-normalization}
and write
\[
 w_\lambda:=w_\lambda(1),\qquad
 X_\lambda(t):=\varphi(x_\lambda(t)),\qquad
 X_\lambda:=X_\lambda(1).
\]
Thus $\varphi$ fixes $N_\rho$, all its values are $Q$-symmetric,
and
\[
 X_\lambda\equiv x_\lambda(1)\pmod{M_{78}(J)}.
\]

\subsection{A family of allowed basis changes}
\label{subsec:root-centralizer}

We construct a seven-parameter family centralizing $N_\rho$.
Its invertible members preserve all preceding normalizations.

Retain $L_R=H\oplus L_\ell\oplus L_s$, $H=H^+\oplus H^-$,
and the signed pair bases of
Lemma~\ref{lemma:signed-transition-blocks}. Put
\[
 C_H(r,s)=rI_{H^+}\oplus sI_{H^-}.
\]
Let $C_\ell(B_\ell)$ denote the repeated block $B_\ell$ on all
long-root pairs, and $C_s(B_s)$ the repeated block $B_s$ on all
short-root pairs. They commute with the Weyl action whenever
\[
 B_\ell P=PB_\ell,\qquad
 B_sP_4=P_4B_s,\qquad B_sQ_4=Q_4B_s.
\]

One additional endomorphism mixes long-root pairs within each
$U_i$. Use the basis $\mathcal B_0$ of $U_0$ from
Subsubsection~\ref{subsubsec:half-sum-weyl} and set
\[
 \Pi_0u_i=\frac14\sum_{j=1}^4u_j,\qquad
 \Pi_0v_i=0\qquad(1\leq i\leq4).
\]
Using the transporters
\[
 n_0=1,\qquad n_1=w_{e_2-e_3},\qquad
 n_2=w_{e_3-e_4}w_{e_2-e_3},
\]
define
\[
 \Pi|_{U_i}=n_i\Pi_0n_i^{-1}\quad(0\leq i\leq2),\qquad
 \Pi|_{H\oplus L_s}=0.
\]
The action table in Subsubsection~\ref{subsubsec:half-sum-weyl}
shows that $\Pi_0$ commutes with the stabilizer of $U_0$ in $N_B$
and with $w_{\beta_0}|_{U_0}$.
Transport gives commutation of $\Pi$ with $N_B$ and, since
$n_1,n_2$ commute with $w_{\beta_0}$, with all of $N_\rho$.
Moreover,
\[
 \Pi^2=\Pi,\qquad \operatorname{rank}\Pi=3,\qquad \Pi Q=Q\Pi,
\]
and $\Pi$ is defined over $\mathbb Z[1/2]$.

For $\mathbf t=(t_1,\ldots,t_7)\in R^7$, put
\[
 B_\ell=(1+t_2)I_2+t_3P,\qquad
 B_s^+=(1+t_4)I_2+t_5P,\qquad
 B_s^-=I_2+t_6P.
\]
Here $B_s$ has blocks $B_s^+,B_s^-$ in the $Q_4$-eigenbasis
\[
 \bigl(e_\beta+e_{\tau\beta},\,-e_{-\beta}-e_{-\tau\beta};\,
       e_\beta-e_{\tau\beta},\,-e_{-\beta}+e_{-\tau\beta}\bigr).
\]
Define
\begin{equation}
\label{eq:root-affine-centralizer}
\begin{aligned}
 C(\mathbf t)
   &=\bigl(C_H(1,1+t_1)\oplus C_\ell(B_\ell)
                          \oplus C_s(B_s)\bigr)+t_7\Pi\\
   &=I_{78}+\sum_{j=1}^7t_jT_j,
 \qquad T_j\in M_{78}(\mathbb Z[1/2]).
\end{aligned}
\end{equation}
The affine dependence on $\mathbf t$ will make the equations for
the change of basis linear.

\begin{lemma}[Allowed changes of basis]
\label{lemma:centralizer-subgroup}
Every $C(\mathbf t)$ commutes with $N_\rho$ and with $Q$.
For $\mathbf t\in J^7$, it is invertible and congruent to $I_{78}$
modulo $M_{78}(J)$. For $\mathbf t\in(J\cap R_0)^7$, it is also
$Q$-symmetric.
\end{lemma}

\begin{proof}
The repeated root-pair blocks commute with $N_\rho$ by
Lemma~\ref{lemma:signed-transition-blocks}; the Cartan block is
scalar on each Weyl-invariant summand $H^\pm$, and commutation
for $\Pi$ was proved above.
All these endomorphisms commute with $Q$.
The remaining assertions follow from the congruence to the identity
and, for $Q$-symmetry, from $\overline{\mathbf t}=\mathbf t$.
\end{proof}

\subsection{The initial diagonal decompositions}
\label{subsec:initial-root-decompositions}

Diagonal commutation determines the blocks in which the unknown
root matrices can have nonzero entries. For every admissible
parameter,
\[
 h_\mu(-1)X_\gamma(t)h_\mu(-1)^{-1}
   =X_\gamma(t)^{(-1)^{\langle\gamma,\mu^\vee\rangle}}.
\]
Put $H_i=h_{[\alpha_i]}(-1)$.
The matrix $X_{e_4}$ commutes with $H_2,H_4,H_1H_3$.
By the character table in
Subsection~\ref{subsec:weyl-diagonal-decompositions}, their
simultaneous eigenspaces are
\begin{equation}
\label{eq:short-initial-blocks}
\begin{gathered}
 H\oplus V_{\pm e_4},\qquad
 U_0\oplus V_{\pm e_3},\qquad
 U_1\oplus V_{\pm e_2},\qquad
 U_2\oplus V_{\pm e_1},\\
 T_{\varepsilon_2\varepsilon_3,+}
   \oplus T_{\varepsilon_2\varepsilon_3,-}
 \qquad(\varepsilon_2,\varepsilon_3\in\{\pm1\}).
\end{gathered}
\end{equation}
The block ranks are $10,12,12,12,8,8,8,8$, leaving
\[
 10^2+3\cdot12^2+4\cdot8^2=788
\]
matrix coordinates. No finer block decomposition is assumed.

For the root-recognition argument in Section~6, we also record
the corresponding decomposition for the long root $e_2-e_3$:
\[
\begin{gathered}
 H\oplus L_{23}\oplus L_{14},\qquad
 L_{12}\oplus L_{13}\oplus L_{24}\oplus L_{34},\\
 V_{\pm e_1}\oplus V_{\pm e_4},\qquad
 V_{\pm e_2}\oplus V_{\pm e_3},\\
 T_{---}\oplus T_{++-},\qquad T_{+-+}\oplus T_{-++},\qquad
 T_{+++}\oplus T_{--+},\qquad T_{+--}\oplus T_{-+-}.
\end{gathered}
\]
Its block ranks are $14,16,8,8,8,8,8,8$, giving $836$ coordinates.
Both decompositions apply to arbitrary root parameters; Weyl
transport gives those for the other roots.

\subsection{Normalization of one short-root matrix}
\label{subsection:short-sector-normalization}

We describe the relations that determine $X_{e_4}$ up to a unique
change of basis from~\eqref{eq:root-affine-centralizer}. Put
\[
 X=X_{e_4},\qquad X^0=x_{e_4}(1),\qquad
 w_4=w_{e_4},\qquad w_0=w_{\beta_0}.
\]
Besides~\eqref{eq:short-initial-blocks}, the matrix $X$ satisfies
\begin{equation}
\label{eq:short-normalization-relations}
\begin{aligned}
 XH_iX&=H_i &&(i=1,3),\\
 Xw_\rho&=w_\rho X
   &&\bigl(\rho=e_i\pm e_j,\ 1\leq i<j\leq3\bigr),\\
 Xw_{e_i}X&=w_{e_i} &&(i=1,2,3),\\
 (w_4X)^3&=I_{78}.&&
\end{aligned}
\end{equation}
The six displayed long-root representatives commute with
$x_{e_4}(1)$, the three coordinate short-root representatives
invert it, and the last equality is the rank-one relation.

Define
\[
 V(X)=w_0Xw_0^{-1},\qquad
 U(X)=(w_4w_0)X(w_4w_0)^{-1}.
\]
For the actual image $X$, the sign conventions give
$V(X)=X_{\beta_0+e_4}$ and $U(X)=X_{\beta_0}$.
The short $A_2$ relations are therefore
\begin{equation}
\label{eq:short-A2-normalization}
 XV=VX,\qquad UV=VU,\qquad XUV=UX,
 \qquad V=V(X),\quad U=U(X).
\end{equation}
Using the first two relations, the last is equivalent to
$[X,U]=V^{-1}$.

Let $F(X)=0$ be the entrywise polynomial system
\eqref{eq:short-normalization-relations}--%
\eqref{eq:short-A2-normalization} in the $788$ allowed coordinates.
Its coefficients are integral, and both $X$ and $X^0$ satisfy it.
Only group relations have been used, not preservation of the
Chevalley bracket.

\begin{lemma}[Normalization of the complete short-root matrix]
\label{lemma:short-sector-rigidity}
\label{lemma:full-short-root-normalization}
Let $X\in\operatorname{GL}_{78}(R)$ have the block form
\eqref{eq:short-initial-blocks}, satisfy $F(X)=0$, and be congruent
to $X^0$ modulo $M_{78}(J)$.
There is a unique $\mathbf t\in J^7$ such that
\[
 C(\mathbf t)X=X^0C(\mathbf t).
\]
If $X$ is $Q$-symmetric, then $\mathbf t\in(J\cap R_0)^7$.
\end{lemma}

\begin{proof}
\smallskip
\noindent\emph{1. Determine the change of basis.}
For signed basis vectors $p,q$, write $[Z]_{p,q}$ for the
coefficient of $p$ in $Zq$.
Define $\mathbf y:M_{78}(R)\to R^7$ by taking, first, the four
entries in the $h_2$-row and the columns
\[
 e_{\alpha_2},\quad -e_{-\alpha_2},\quad
 e_{\alpha_5},\quad -e_{-\alpha_5},
\]
and then, with $p=-e_{\alpha_3+\alpha_5}$,
\[
 y_5(Z)=[Z]_{p,e_{\alpha_3}},\qquad
 y_6(Z)=[Z]_{p,-e_{-\alpha_3}},\qquad
 y_7(Z)=[Z]_{p,g^{-1}e_{\alpha_3}}.
\]
Here $g$ is the fixed Weyl word used in
Subsubsection~\ref{subsubsec:half-sum-weyl}.
The vector $g^{-1}e_{\alpha_3}$ is a signed root vector in $L_{12}$;
all seven entries are distinct allowed coordinates, up to sign.

Using $C(\mathbf t)=I_{78}+\sum_jt_jT_j$, define the
$7\times7$ matrix $M(X)$ by
\[
 M(X)_{\cdot j}=\mathbf y(T_jX-X^0T_j).
\]
The seven equations
\begin{equation}
\label{eq:seven-linear-intertwining}
 \mathbf y\bigl(C(\mathbf t)X-X^0C(\mathbf t)\bigr)=0
\end{equation}
are the linear system
\[
 M(X)\mathbf t=-\mathbf y(X-X^0).
\]
At the standard matrix, direct multiplication gives
\[
 M(X^0)\mathbf t=
 \begin{pmatrix}
 (t_5+t_6)/2\\
 (t_4-t_1)/2\\
 (t_5-t_6)/2\\
 (t_4+t_1)/2\\
 t_4-t_2-t_7/8\\
 -t_3-t_7/8\\
 -t_7/8
 \end{pmatrix},
 \qquad
 \det M(X^0)=\frac1{32}.
\]
Since $M(X)\equiv M(X^0)\pmod{M_7(J)}$, the matrix $M(X)$
is invertible and the system has a unique solution
$\mathbf t\in J^7$.
The resulting $C=C(\mathbf t)$ is invertible and commutes with
the diagonal involutions and all constant matrices occurring in $F$.
Consequently,
\[
 X^*:=C^{-1}X^0C
\]
has the required block form, satisfies $F(X^*)=0$, and is congruent
to $X^0$ modulo $M_{78}(J)$.

\smallskip
\noindent\emph{2. Prove the full intertwining identity.}
The Jacobian $DF(X^0)$ has $788$ columns and $11402$ nonzero rows.
The exact calculation selects $781$ polynomial components
$F_{\rm sel}$ such that, after deleting the seven columns
underlying $\mathbf y$, their Jacobian gives a square matrix
$J_{\rm sq}$ with
\begin{equation}
\label{eq:short-normalization-minor}
 |\det J_{\rm sq}|=2^{200}.
\end{equation}
This certificate is contained in
\texttt{sections/section4\_complete\_short\_root/}
of \texttt{twisted\_e6\_verification.zip}.
The script \texttt{compute\_certificate.py} constructs the system
and the certificate. The coordinate and equation indexing is
recorded in \texttt{data/coordinates.json}, the row and column
selection in \texttt{data/selection.json}, and the exact minor
in \texttt{data/selected\_minor.npz}.

Let $J_{\rm sel}$ be the $781\times788$ Jacobian before deleting
columns, and let $Y_0$ represent $\mathbf y$ on the allowed
coordinates.
Each row of $Y_0$ selects one of the seven distinct coordinates
with coefficient $\pm1$.
Expanding along these rows and using
\eqref{eq:short-normalization-minor} gives
\[
 \mathcal M_0:=
 \begin{pmatrix}J_{\rm sel}\\Y_0\end{pmatrix}
 \in\operatorname{GL}_{788}(\mathbb Z[1/2]).
\]

Both $X$ and $X^*$ satisfy the selected polynomial equations.
Moreover, $CX^*=X^0C$, so
\eqref{eq:seven-linear-intertwining} gives
$\mathbf y(C(X-X^*))=0$.
Let $z\in J^{788}$ be the column of entries of $X-X^*$.
Since $X,X^*\equiv X^0$ and $C\equiv I_{78}$ modulo $J$,
polynomial finite differences give
\[
 0=
 \begin{pmatrix}
 F_{\rm sel}(X)-F_{\rm sel}(X^*)\\
 \mathbf y(C(X-X^*))
 \end{pmatrix}
 =(\mathcal M_0+\mathcal E)z,
 \qquad \mathcal E\in M_{788}(J).
\]
Indeed, factoring each polynomial difference gives coefficients
reducing to the corresponding derivatives at $X^0$; the lower
block reduces to $Y_0$.
The coefficient matrix is invertible over $R$, so $z=0$.
Hence $X=X^*$ and $CX=X^0C$.
This is a finite argument; it uses neither nilpotence of $J$
nor completeness or Henselianity.

Any intertwiner in the family satisfies the seven linear equations,
so their uniqueness proves uniqueness of $\mathbf t$.
Finally, when $X$ is $Q$-symmetric, applying $\sigma$ to the full
intertwining identity gives the same identity for
\[
 \sigma(C(\mathbf t))=C(\overline{\mathbf t}).
\]
Uniqueness implies $\overline{\mathbf t}=\mathbf t$.
\end{proof}

\subsection{Completion of the normalization}
\label{subsec:root-normalization-completion}

Once the short roots are normalized, the long roots are determined
by their squares.

\begin{lemma}[Uniqueness from the square]
\label{lemma:root-square-uniqueness}
Let $U,Y\in\operatorname{GL}_n(R)$ satisfy
\[
 (U-I_n)^3=0,\qquad
 Y\equiv U\pmod{M_n(J)},\qquad Y^2=U^2.
\]
Then $Y=U$.
\end{lemma}

\begin{proof}
Writing $U=I_n+N$, with $N^3=0$, gives
\[
 U=I_n+\frac12(Y^2-I_n)-\frac18(Y^2-I_n)^2.
\]
Thus $U$ commutes with $Y$, and $(Y-U)(Y+U)=0$.
Since $Y+U$ is congruent to the invertible matrix $2U$ modulo
$M_n(J)$, it is invertible. Hence $Y=U$.
\end{proof}

\begin{proposition}[Constant-root normalization]
\label{prop:root-unipotent-normalization}
There is a matrix $C\in\operatorname{GL}_{78}(R)^\sigma$,
congruent to $I_{78}$ modulo $M_{78}(J)$ and commuting with
$N_\rho$, such that
\[
 C\varphi(x_\lambda(1))C^{-1}=x_\lambda(1)
 \qquad(\lambda\in\Phi_\rho).
\]
\end{proposition}

\begin{proof}
Apply Lemma~\ref{lemma:full-short-root-normalization} to $X_{e_4}$.
It gives a $Q$-symmetric $C=C(\mathbf t)$ commuting with $N_\rho$.
Replace $\varphi$ by $i_C\circ\varphi$, retaining the notation
$X_\lambda$ for its current images.
Then $X_{e_4}=x_{e_4}(1)$ on the whole module, and $N_\rho$
remains fixed.
Transitivity of the Weyl group on short roots gives
\[
 X_\gamma=x_\gamma(1)\qquad(\gamma\in\Phi_{\rho,s});
\]
if transport gives parameter $-1$, use inversion.

Every long root is $\lambda=\gamma+\delta$ for two orthogonal
short roots, and
\[
 [x_\gamma(1),x_\delta(1)]
   =x_\lambda(2\varepsilon),
 \qquad \varepsilon\in\{1,-1\}.
\]
Applying $\varphi$ and inverting when necessary yields
\[
 X_\lambda^2=x_\lambda(1)^2.
\]
The long-root matrix $x_\lambda(1)$ is an absolute root matrix,
so $(x_\lambda(1)-I_{78})^3=0$.
Its standard residue and Lemma~\ref{lemma:root-square-uniqueness}
therefore give $X_\lambda=x_\lambda(1)$.
\end{proof}

\subsection{The accumulated basis change}
\label{subsec:accumulated-basis-change}

We collect the normalizations and identify the input for the
bracket correction.
Let $\varphi_1$ be the automorphism from
Proposition~\ref{prop:reduction}.
The changes in Sections~3--4 combine to give
\[
 C_0\in\operatorname{GL}_{78}(R)^\sigma,\qquad
 C_0\equiv I_{78}\pmod{M_{78}(J)},\qquad
 \psi_0:=i_{C_0}\circ\varphi_1,
\]
where
\[
 \psi_0(x_\lambda(1))=x_\lambda(1)
 \qquad(\lambda\in\Phi_\rho).
\]
Their Weyl and diagonal words are therefore fixed as well.

In particular,
\begin{equation}
\label{eq:accumulated-frame-transporter}
 C_0^{-1}x_\lambda(1)C_0
   =\varphi_1(x_\lambda(1))\in E.
\end{equation}
The matrices on the right preserve the standard Chevalley bracket.
It follows that the transported bracket
\[
 \{u,v\}_{C_0}:=C_0\{C_0^{-1}u,C_0^{-1}v\}
\]
is invariant under every standard $x_\lambda(1)$.

The constructions in Sections~3--4 use only group relations and
standard reductions, not injectivity or surjectivity of the map.
Thus the same linear normalization is available for homomorphisms
with these properties and a local target.

In Section~5 we use the invariant bracket to replace $C_0$ by
$C=KC_0\in G$, where $K$ centralizes the standard frame.
The map $\psi=i_C\circ\varphi_1$ will therefore retain all these
normalizations.


\bigskip

\section{Correction of the accumulated basis change}
\label{sec:correction-of-basis-change}

We replace the matrix $C_0$ from
Subsection~\ref{subsec:accumulated-basis-change} by an element of
the adjoint twisted group without changing the normalized frame.
Invariant tensors determine how $C_0$ changes the Chevalley bracket;
the local bracket stabilizer places a correction in the group;
and the trivial frame centralizer gives uniqueness and descent.
We establish these three ingredients before carrying out the
correction.

Throughout this section, put
\[
 A=\mathbb Z[1/2],\qquad
 \mathcal G=\mathbf G_{\mathrm{ad}}(E_6).
\]
Let $\{\ ,\ \}$ denote the Chevalley bracket on $L_A$ and its base
changes. Both $Q$ and the matrices of $\mathcal G$ preserve this
bracket; see Section~1 and
\cite{Chevalley,HumphreysLie,SteinbergLectures}.
For a commutative $A$-algebra $S$, write
\[
 L_S=L_A\otimes_A S,\qquad
 \mathcal F_S=(x_\lambda(1))_{\lambda\in\Phi_\rho}.
\]

\subsection{Invariant alternating tensors}
\label{subsec:invariant-alternating-tensors}

We determine the alternating tensors preserved by the constant
frame. The answer consists of the three components of the standard
bracket with respect to the $Q$-eigenspaces.

Put
\[
 \pi_\pm=\frac{I_{78}\pm Q}{2},\qquad
 L_S^\pm=\pi_\pm L_S.
\]
These summands have ranks $52$ and $26$.
For $u_\pm=\pi_\pm u$ and $v_\pm=\pi_\pm v$, define
\[
\begin{aligned}
 B_{++}(u,v)&=\{u_+,v_+\},\\
 B_{+-}(u,v)&=\{u_+,v_-\}+\{u_-,v_+\},\\
 B_{--}(u,v)&=\{u_-,v_-\}.
\end{aligned}
\]
Their sum is the standard bracket. Each is frame-invariant,
since the frame matrices preserve the bracket and commute with $Q$.

\begin{lemma}[Invariant tensors]
\label{lemma:invariant-alternating-tensors}
For every commutative $A$-algebra $S$, the alternating $S$-bilinear
maps $\Lambda^2L_S\to L_S$ invariant under $\mathcal F_S$ form
\[
 SB_{++}\oplus SB_{+-}\oplus SB_{--}.
\]
This description is preserved by arbitrary base change.
\end{lemma}

\begin{proof}
It suffices to impose invariance under
$x_{\pm[\alpha_i]}(1)$, $1\leq i\leq4$:
their words give the simple Weyl representatives and diagonal
involutions, and Weyl conjugation gives the other frame elements.

Work over $A$ in a $Q$-eigenbasis adapted to the four diagonal
involutions. For an alternating tensor with coefficients $b_{ij}^k$,
invariance under a matrix $g$ means
\[
 \sum_l g_{kl}b_{ij}^l
 -\sum_{p<q}(g_{pi}g_{qj}-g_{qi}g_{pj})b_{pq}^k=0
 \qquad(i<j).
\]
Diagonal invariance removes coefficients with incompatible
characters. Since the frame commutes with $Q$, the remaining
system splits into six channels. Exact computation gives the
following ranks and kernels, with invertible square subsystems
over $A$ certifying the ranks:
\[
\begin{array}{c|r|r|c}
 \text{channel}&\text{unknowns}&\text{rank}&\text{kernel}\\ \hline
 \Lambda^2L_A^+\longrightarrow L_A^+&5016&5015&AB_{++}\\
 \Lambda^2L_A^+\longrightarrow L_A^-&1932&1932&0\\
 L_A^+\otimes L_A^-\longrightarrow L_A^+&3968&3968&0\\
 L_A^+\otimes L_A^-\longrightarrow L_A^-&2320&2319&AB_{+-}\\
 \Lambda^2L_A^-\longrightarrow L_A^+&1108&1107&AB_{--}\\
 \Lambda^2L_A^-\longrightarrow L_A^-&506&506&0.
\end{array}
\]
The systems, standard tensors, and selected minors are in
\texttt{sections/section5\_invariant\_tensors/}
of \texttt{twisted\_e6\_verification.zip}.
The script \texttt{compute\_certificate.py} constructs them, and
\texttt{verify\_certificate.py} checks the stored certificates.
The variable orderings, selected rows, and retained columns are
recorded in \texttt{data/}.

In each nonzero channel, the standard tensor solves the full
system and has a unit coefficient in the sole free coordinate.
For any solution over $S$, subtract its unique multiple matching
that coordinate. The selected invertible system forces all
remaining coefficients to vanish.
In each zero channel, an invertible square subsystem forces the
entire tensor to vanish. All the inverses are defined over $A$,
so no flatness assumption on $S$ is needed.
\end{proof}

\subsection{The frame centralizer}
\label{subsec:frame-centralizer}

The linear centralizer describes the available corrections.
Its intersection with the adjoint group is trivial, which will
make the corrected conjugating element unique over every
coefficient ring.

\begin{lemma}[The frame centralizer]
\label{lemma:graph-matrix-not-inner}
\label{lemma:frame-centralizer}
For every commutative $A$-algebra $S$,
\[
\begin{aligned}
 \operatorname{Cent}_{\operatorname{End}_S(L_S)}(\mathcal F_S)
   &=S\pi_+\oplus S\pi_-,\\
 \operatorname{Cent}_{\mathcal G(S)}(\mathcal F_S)
   &=\{I_{78}\}.
\end{aligned}
\]
Moreover, $Q\notin\mathcal G(S)$ whenever $S\ne0$.
\end{lemma}

\begin{proof}
We first establish the generation identities
\begin{equation}
\label{eq:bracket-generation}
 \langle\{L_A^+,L_A^+\}\rangle_A=L_A^+,\qquad
 \langle\{L_A^+,L_A^-\}\rangle_A=L_A^-.
\end{equation}
The root-vector parts are generated by
\[
\begin{aligned}
 \{h_\alpha,e_\alpha\}&=2e_\alpha
       &&(\tau\alpha=\alpha),\\
 \{h_\alpha+h_{\tau\alpha},
       e_\alpha\pm e_{\tau\alpha}\}
   &=2(e_\alpha\pm e_{\tau\alpha})
       &&(\tau\alpha\ne\alpha),
\end{aligned}
\]
using orthogonality in the second line.
The Cartan parts are generated by
\[
\begin{aligned}
 \{e_{\alpha_i},e_{-\alpha_i}\}&=h_i
       &&(i=3,4),\\
 \{e_{\alpha_i}+e_{\tau\alpha_i},
       e_{-\alpha_i}\pm e_{-\tau\alpha_i}\}
   &=h_i\pm h_{\tau\alpha_i}
       &&(i=1,2).
\end{aligned}
\]
Since only $2$ is inverted, these identities hold after every
base change.

If $D\in\operatorname{End}_S(L_S)$ centralizes the frame, then
$(u,v)\mapsto D\{u,v\}$ is frame-invariant.
By Lemma~\ref{lemma:invariant-alternating-tensors}, it equals
$aB_{++}+bB_{+-}+cB_{--}$.
Its restrictions to $L_S^+\times L_S^+$ and
$L_S^+\times L_S^-$, together with
\eqref{eq:bracket-generation}, give
\[
 D|_{L_S^+}=aI,\qquad D|_{L_S^-}=bI.
\]
Thus $D=a\pi_++b\pi_-$.
Conversely, these endomorphisms centralize the frame because
the frame commutes with $Q$.

Next, $Q$ does not belong to $\mathcal G(F)$ for any field $F$
with $2\in F^\times$.
Indeed, after extension to an algebraic closure, its action on
the standard maximal torus would be induced by an element of
$W(E_6)$. But this action is the nontrivial diagram automorphism:
it preserves the fundamental chamber, whose Weyl-group stabilizer
is trivial. Passing to a residue field therefore also excludes
$Q$ from $\mathcal G(S)$ for every nonzero $S$.

Finally, let $D=a\pi_++b\pi_-\in\mathcal G(S)$.
Invertibility makes $a,b$ units.
Bracket preservation and~\eqref{eq:bracket-generation} give
$a^2=a$, hence $a=1$.
The identity
\[
 \{e_{\alpha_1}-e_{\alpha_6},
       e_{-\alpha_1}-e_{-\alpha_6}\}=h_1+h_6
\]
then gives $b^2=1$.
Thus $e_-=(1-b)/2$ is an idempotent, and on the factor $e_-S$
the matrix $D$ equals $Q$.
The exclusion just proved forces $e_-S=0$, so $b=1$ and $D=I_{78}$.
\end{proof}

\subsection{The local stabilizer of the Chevalley bracket}
\label{subsec:local-bracket-stabilizer}

We now show that a bracket-preserving matrix congruent to the
identity belongs to the adjoint group.
Let $\mathcal H\subseteq\operatorname{GL}_{78,A}$ be the closed
subgroup scheme defined by
\[
 \mathcal H(S)=
 \{D\in\operatorname{GL}_{78}(S)\mid
       D\{u,v\}=\{Du,Dv\}\text{ for all }u,v\in L_S\}.
\]
The faithful Chevalley representation gives a closed immersion
$\mathcal G\subseteq\mathcal H$.
We compare these schemes near the identity by computing their
tangent modules.

\begin{lemma}[Derivations of the Chevalley module]
\label{lemma:bracket-stabilizer-tangent}
For every commutative $A$-algebra $S$,
\[
 \operatorname{Der}_S(L_S)
   =\bigoplus_{\beta\in\Phi}S\,\operatorname{ad}e_\beta
      \ \oplus\ \bigoplus_{i=1}^6S\,\partial_i,
\]
where, for $\beta=\sum_i m_i(\beta)\alpha_i$,
\[
 \partial_i(h_j)=0,\qquad
 \partial_i(e_\beta)=m_i(\beta)e_\beta.
\]
This is the image of $\operatorname{Lie}(\mathcal G)_S$ in
$\operatorname{End}_S(L_S)$. Hence the tangent module of
$\mathcal H$ at the identity is free of rank $78$ after every
base change.
\end{lemma}

\begin{proof}
For every nonzero
\[
 \delta\in(\Phi\cup\{0\})-(\Phi\cup\{0\}),
\]
we first show that some $h\in H_A$ satisfies $\delta(h)=1$.
Normalize roots to have square length $2$.
For a root $\delta$, its coroot has value $2$.
Otherwise write $\delta=\alpha-\beta$ with distinct roots.
Then
\[
 \delta(h_\alpha)=2-(\alpha,\beta)\in\{1,2,3,4\}.
\]
Only the value $3$ requires an argument.
In that case $(\delta,\delta)=6$.
If every $\delta(h_i)$ were divisible by $3$, then $\delta/3$
would be a weight of square length $2/3$.
This is impossible: the inverse Cartan matrix has positive
entries and diagonal
\[
 \left(\frac43,\frac{10}3,6,2,\frac{10}3,\frac43\right),
\]
so every nonzero dominant weight, and hence every nonzero weight,
has square length at least $4/3$.
Thus the values $\delta(h_i)$ generate the unit ideal in $A$,
proving the claim.

Use the root-lattice grading
$(L_S)_0=H_S$, $(L_S)_\beta=Se_\beta$.
Every derivation is a sum of homogeneous derivations, with degrees
among the differences considered above.
For a derivation $D_\delta$ of nonzero degree, applying it to
$\{h,v\}=\beta(h)v$ gives
\[
 \operatorname{ad}(D_\delta h)=-\delta(h)D_\delta.
\]
Choose $h$ with $\delta(h)=1$.
If $\delta\notin\Phi$, then $D_\delta h=0$ and $D_\delta=0$.
Otherwise, $D_\delta h\in Se_\delta$, so $D_\delta$ is a multiple
of $\operatorname{ad}e_\delta$.

For degree zero, write $D_0e_\beta=c_\beta e_\beta$.
The derivation identity on $\{h,e_\beta\}$ gives
$\beta(D_0h)=0$ for every root.
Also,
\[
 D_0h_i=(c_{\alpha_i}+c_{-\alpha_i})h_i.
\]
Applying $\alpha_i$ and using the unit $2$ shows that $D_0h_i=0$.
Thus $D_0H_S=0$, and the identities on opposite root vectors
give $c_{-\beta}=-c_\beta$.
Since the nonzero root-string constants are $\pm1$,
\[
 c_{\alpha+\beta}=c_\alpha+c_\beta
 \qquad(\alpha,\beta,\alpha+\beta\in\Phi).
\]
Induction on positive-root height now gives
$c_\beta=\sum_i m_i(\beta)c_{\alpha_i}$, so
$D_0=\sum_i c_{\alpha_i}\partial_i$.

The displayed derivations are independent: the root derivations
have distinct nonzero degrees and send $h_\beta$ to $-2e_\beta$,
while the $\partial_i$ are distinguished on the simple root vectors.
They are the tangent vectors of the absolute root subgroups and
of the adjoint torus
$\operatorname{Hom}(\mathbb Z\Phi,\mathbb G_m)$.
Finally, substituting $I_{78}+\varepsilon D$, with $\varepsilon^2=0$,
in the equations of $\mathcal H$ identifies its tangent module
with $\operatorname{Der}_S(L_S)$.
\end{proof}

The toral derivations here are $\partial_i$; they need not be
generated by the $\operatorname{ad}h_i$ when $3$ is not invertible.

\begin{proposition}[Local bracket stabilizer]
\label{prop:local-bracket-stabilizer}
Let $T$ be a local commutative $A$-algebra with maximal ideal
$\mathfrak n$. If
\[
 D\in\mathcal H(T),\qquad
 D\equiv I_{78}\pmod{M_{78}(\mathfrak n)},
\]
then $D\in\mathcal G(T)$.
\end{proposition}

\begin{proof}
Put $\mathfrak p=\ker(A\to T/\mathfrak n)$, and let $e$ be the
point of the identity section over $\mathfrak p$.
By Lemma~\ref{lemma:bracket-stabilizer-tangent}, the Jacobian
of the bracket equations at $e$ has rank
\[
 78^2-78=6006.
\]
Choose $6006$ equations with a $6006\times6006$ Jacobian minor
$\Delta$ nonzero at $e$.
On the open locus
$\mathcal V\subseteq\operatorname{GL}_{78,A}$ where $\Delta$
is invertible, their zero scheme $\mathcal Y$ is smooth over $A$
of relative dimension $78$; see
\cite[Tags 00T7 and 00T8]{StacksProject}.
There are closed immersions
\[
 \mathcal G\cap\mathcal V
   \ \subseteq\ \mathcal H\cap\mathcal V
   \ \subseteq\ \mathcal Y.
\]

Both $\mathcal G$ and $\mathcal Y$ are smooth of relative
dimension $78$ over the regular Noetherian ring $A$.
Since $e$ is a rational point of the fibre over $\mathfrak p$,
their local rings at $e$ are regular of dimension
$78+\dim A_{\mathfrak p}$.
The induced surjection
\[
 \mathcal O_{\mathcal Y,e}\longrightarrow
 \mathcal O_{\mathcal G,e}
\]
has prime kernel, because the target is a domain.
Equality of dimensions makes this prime have height zero;
the source is a domain, so the kernel is zero.

The ideal defining $\mathcal G\cap\mathcal V$ in the Noetherian
scheme $\mathcal Y$ is finitely generated. Its stalk at $e$ is zero,
so it vanishes on an open neighborhood $\mathcal U$ of $e$.
On this neighborhood all three schemes coincide.
In particular, $\mathcal U$ is an open neighborhood of $e$
in $\mathcal H$ contained in $\mathcal G$.

The morphism $\operatorname{Spec}T\to\mathcal H$ defined by $D$
sends the closed point to $e$.
The inverse image of $\mathcal U$ is an open subset containing
that closed point, hence is all of $\operatorname{Spec}T$.
Thus $D\in\mathcal G(T)$.
The Noetherian argument takes place over $A$, not over $T$.
\end{proof}

\subsection{Correction and descent}
\label{subsec:transported-Chevalley-bracket}
\label{subsec:correction-of-transported-bracket}
\label{subsec:descent-of-correction}

We now combine these results.
The formulation below uses only the normalized finite family,
so it applies both to the automorphism under consideration and
to the homomorphism normalizations described at the end of
Section~4.

\begin{proposition}[Correction and descent]
\label{prop:descent-of-correction}
Suppose
\[
 C_0\in\operatorname{GL}_{78}(R)^\sigma,\qquad
 C_0\equiv I_{78}\pmod{M_{78}(J)},
\]
and every matrix
\[
 a_\lambda=C_0^{-1}x_\lambda(1)C_0
 \qquad(\lambda\in\Phi_\rho)
\]
preserves the Chevalley bracket.
There is a unique $C\in\mathcal G(R)$ satisfying
\[
 Ca_\lambda C^{-1}=x_\lambda(1)
 \qquad(\lambda\in\Phi_\rho).
\]
Moreover,
\[
 C=(r\pi_++z\pi_-)C_0
   \in G_{\mathrm{ad},\sigma}(E_6,R),\qquad
 C\equiv I_{78}\pmod{M_{78}(J)}
\]
for units $r,z\in(1+J)\cap R_0$.
\end{proposition}

\begin{proof}
\smallskip
\noindent\emph{1. Determine the transported bracket.}
Put
\[
 B(u,v)=\{u,v\}_{C_0}
       :=C_0\{C_0^{-1}u,C_0^{-1}v\},\qquad
 \sigma_L(u)=Q\overline u.
\]
This is an alternating $R$-bilinear Lie bracket.
The assumptions on $C_0$ and $a_\lambda$ give
\begin{equation}
\label{eq:transported-bracket-properties}
\begin{aligned}
 B(u,v)&\equiv\{u,v\}\pmod{JL_R},\\
 B(x_\lambda(1)u,x_\lambda(1)v)
   &=x_\lambda(1)B(u,v)
       &&(\lambda\in\Phi_\rho),\\
 B(\sigma_Lu,\sigma_Lv)&=\sigma_LB(u,v).
\end{aligned}
\end{equation}
The last identity uses $C_0\sigma_L=\sigma_LC_0$, not ordinary
commutation with $Q$.

Lemma~\ref{lemma:invariant-alternating-tensors}, together with
reduction and semilinear compatibility, yields
\[
 B=rB_{++}+sB_{+-}+tB_{--},
 \qquad r,s,t\in(1+J)\cap R_0.
\]
Jacobi on $u_+,v_+\in L_R^+$ and $w_-\in L_R^-$ gives
\[
 s(s-r)\,\bigl\{\{u_+,v_+\},w_-\bigr\}=0.
\]
For
\[
 u_+=h_1+h_6,\qquad
 v_+=e_{\alpha_1}+e_{\alpha_6},\qquad
 w_-=e_{-\alpha_1}-e_{-\alpha_6},
\]
the last bracket is $2(h_1-h_6)$.
Since $2s$ is a unit, $s=r$. Thus
\begin{equation}
\label{eq:two-bracket-coefficients}
 \{\ ,\ \}_{C_0}=rB_{++}+rB_{+-}+tB_{--},
 \qquad r,t\in(1+J)\cap R_0.
\end{equation}

\smallskip
\noindent\emph{2. Correct the bracket over a local extension.}
Set
\[
 R'=R[Z]/(Z^2-rt),\qquad
 \mathfrak m=(J,Z-1)R',\qquad
 S=R'_{\mathfrak m},\qquad \mathfrak n=\mathfrak mS,
\]
and let $z$ be the image of $Z$.
Since $rt\equiv1\pmod J$, the ideal $\mathfrak m$ is maximal and
$S/\mathfrak n\simeq R/J$.
The algebra $R'$ is free of rank two over $R$, so $R\to S$ is
flat and local, hence faithfully flat.
The involution extends by $\overline Z=Z$ and preserves
$\mathfrak m$. Consequently,
\[
 z\in S_0^\times,\qquad z\equiv1\pmod{\mathfrak n},\qquad z^2=rt.
\]

Put
\[
 K=r\pi_++z\pi_-,\qquad C=KC_0.
\]
The matrix $K$ is $Q$-symmetric, congruent to the identity, and
centralizes the frame because it is a polynomial in $Q$.
Transporting~\eqref{eq:two-bracket-coefficients} once more gives
\[
 \{\ ,\ \}_{C}
   =B_{++}+B_{+-}+\frac{rt}{z^2}B_{--}
   =\{\ ,\ \}.
\]
Thus $C\in\mathcal H(S)$ and
$C\equiv I_{78}\pmod{M_{78}(\mathfrak n)}$.
Proposition~\ref{prop:local-bracket-stabilizer} gives
$C\in\mathcal G(S)$; its $Q$-symmetry gives
$C\in G_{\mathrm{ad},\sigma}(E_6,S)$.
Since $K$ centralizes the frame,
\[
 Ca_\lambda C^{-1}=x_\lambda(1)
 \qquad(\lambda\in\Phi_\rho).
\]

\smallskip
\noindent\emph{3. Descend the correction.}
Let $C_1,C_2$ be the two pullbacks of $C$ to $S_2=S\otimes_R S$.
Both have the same conjugacies with the matrices $a_\lambda$,
which are defined over $R$.
Hence $C_2C_1^{-1}\in\mathcal G(S_2)$ centralizes $\mathcal F_{S_2}$.
Lemma~\ref{lemma:frame-centralizer} gives $C_1=C_2$, without
requiring $S_2$ to be local.
The fpqc sheaf property
\cite[Tag 023Q]{StacksProject} therefore descends $C$ uniquely
to $\mathcal G(R)$.

It follows that $K=CC_0^{-1}$ is defined over $R$.
Over $S$,
\[
 K(h_1-h_6)=z(h_1-h_6),
\]
so its $h_1$-coefficient shows that $z$ also comes from $R$.
Faithful flatness descends the identities
$\overline z=z$ and $z^2=rt$.
The latter makes $z$ a unit, and
$\mathfrak n\cap R=J$ gives $z-1\in J$.
Thus $K$ and $C$ are $Q$-symmetric and congruent to the identity
over $R$, proving the stated form of $C$.

Finally, if $C'\in\mathcal G(R)$ has the same frame conjugacies,
then $C'C^{-1}$ centralizes the frame.
Lemma~\ref{lemma:frame-centralizer} gives $C'=C$.
\end{proof}

The correction is compatible with arbitrary equivariant base change.
Indeed, for an equivariant ring homomorphism $f:R\to T$,
the matrix $f_*(C)$ satisfies the base-changed frame conjugacies
and is their unique solution in $\mathcal G(T)$ by
Lemma~\ref{lemma:frame-centralizer}.
No locality assumption on $T$ is needed.
This concerns the finite matrix family, not an extension of the
abstract homomorphism to a group over $T$.

For the matrix $C_0$ constructed in Section~4, the hypotheses hold
because
\[
 a_\lambda=\varphi_1(x_\lambda(1))\in E.
\]
As $E$ is normal in $G_{\mathrm{ad},\sigma}(E_6,R)$, the map
\[
 \psi:=i_C\circ\varphi_1=i_K\circ\psi_0
\]
is now an automorphism of $E$.
It fixes every $x_\lambda(1)$ and hence all their Weyl and diagonal
words. In Section~6 we determine its action on root elements
with arbitrary parameters.


\section{Root parameters and completion}
\label{sec:root-parameters-and-completion}

We recover the coefficient map of a homomorphism fixing the
standard frame. Applying this to the normalized automorphism
$\psi=i_C\circ\varphi_1$ from
Proposition~\ref{prop:descent-of-correction} completes the adjoint
classification. We then treat the simply connected form and finish
with relative normalization for homomorphisms with a local target.

For a commutative ring $U$ with involution and $2\in U^\times$, put
\[
 E(U):=E'_{\mathrm{ad},\tau\theta_U}(E_6,U),\qquad
 G(U):=G_{\mathrm{ad},\tau\theta_U}(E_6,U),\qquad
 U^-:=\{u\in U\mid\overline u=-u\}.
\]
Thus $U=U_0\oplus U^-$; the symbols $E,G$ without arguments
continue to refer to $R$.

\subsection{Two root-recognition calculations}
\label{subsec:long-root-parameters}
\label{subsec:root-recognition}

We recognize long-root elements and short-root elements with skew
parameters inside the adjoint group by commutation with finite
families of constant matrices. Bracket preservation eliminates
the extra solutions of the linear commutation equations.
Retain $A=\mathbb Z[1/2]$ and
$\mathcal G=\mathbf G_{\mathrm{ad}}(E_6)$ from Section~5. Set
\[
 \alpha=\alpha_4=e_2-e_3,\qquad
 \beta=\alpha_2+\alpha_3,\qquad
 \tau\beta=\alpha_3+\alpha_5,\qquad [\beta]=e_3,
\]
\[
 N_\alpha=\operatorname{ad}e_\alpha,\qquad
 N_-=\operatorname{ad}e_\beta-\operatorname{ad}e_{\tau\beta}.
\]
Both operators have cube zero. Over any $A$-algebra $T$, put
\[
\begin{aligned}
 U_\alpha(c)&=x_\alpha(c)
   =I_{78}+cN_\alpha+\frac{c^2}{2}N_\alpha^2,\\
 U_-(c)&=x_\beta(c)x_{\tau\beta}(-c)
   =I_{78}+cN_-+\frac{c^2}{2}N_-^2.
\end{aligned}
\]
The second formula uses orthogonality of $\beta,\tau\beta$.
With an involution on $T$, these are twisted root elements
precisely for $c\in T_0$ and $c\in T^-$, respectively.

For $N=N_\alpha$ or $N_-$, define, with $\mu,\lambda\in\Phi_\rho$,
\[
 \mathcal S_N=
 \{h_\mu(-1)\mid h_\mu(-1)N=Nh_\mu(-1)\}
 \ \cup
 \{x_\lambda(1)\mid x_\lambda(1)N=Nx_\lambda(1)\}.
\]
These families are selected over $A$ and then base changed;
additional equalities over $T$ do not enlarge them.
Every member commutes with the corresponding $U_\alpha(c)$ or
$U_-(c)$ for all $c$.

\begin{lemma}[Root recognition]
\label{lemma:long-root-rigidity}
\label{lemma:skew-root-rigidity}
For every commutative $A$-algebra $T$,
\[
\begin{aligned}
 \operatorname{Cent}_{\mathcal G(T)}(\mathcal S_{N_\alpha})
   &=\{U_\alpha(c)\mid c\in T\},\\
 \operatorname{Cent}_{\mathcal G(T)}(\mathcal S_{N_-})
   &=\{U_-(c)\mid c\in T\}.
\end{aligned}
\]
The parameter is unique in both cases. If $T$ has an involution
and a recognized element $Y$ belongs to $G(T)$, its parameter lies
in $T_0$ in the first case and in $T^-$ in the second.
\end{lemma}

\begin{proof}
\smallskip
\noindent\emph{1. Solve the linear commutation equations.}
The diagonal decompositions of
Subsection~\ref{subsec:initial-root-decompositions} leave $836$
coordinates in the long-root case and $788$ in the short-root case.
Each family contains $33$ constant root matrices. Exact elimination
over $A$ gives, respectively,
\begin{align}
 Y&=aI_{78}+bQ+cN_\alpha+dQN_\alpha+eN_\alpha^2+uT_\alpha,
       \qquad 3u=0,
       \label{eq:long-recognition-linear}\\
 Y&=aI_{78}+bQ+cN_-+dQN_-+pN_-^2+\ell QN_-^2+uT_-.
       \label{eq:skew-recognition-linear}
\end{align}
Here $T_\alpha$ is zero except for
\[
 T_\alpha(e_{-\alpha})=-h_1+h_2-h_5+h_6,
\]
and the integral matrix $T_-$ is normalized by
\[
\begin{aligned}
 T_-(h_1)&=3(e_\beta-e_{\tau\beta}),\\
 T_-(e_{-\beta})&=2h_1+h_2-h_5-2h_6,\\
 T_-(e_{-\tau\beta})&=-2h_1-h_2+h_5+2h_6.
\end{aligned}
\]
The certificates select square subsystems of sizes $830$ and $781$
with determinants in $A^\times$.
The displayed model matrices solve these selected equations, and
the matrices taking their coefficients to the six or seven free
coordinates have determinants $-4$ and $-8$, respectively.
Thus the selected equations force the displayed forms, with unique
coefficients, after arbitrary base change.
In the long-root case, substitution in the remaining equations
gives exactly $\pm3u=0$; in the short-root case all remaining
equations vanish. In particular, the term $uT_\alpha$ cannot yet
be removed.

The systems, coordinate maps, and selected minors are in
\texttt{sections/section6\_parameter\_rigidity/}
of \texttt{twisted\_e6\_verification.zip}.
The full matrix $T_-$ is stored in \texttt{data/short\_T\_minus.npz}.
The script \texttt{compute\_certificate.py} reconstructs both
systems and checks the bracket identities used below.

\smallskip
\noindent\emph{2. Impose bracket preservation.}
For either linear form,
\begin{equation}
\label{eq:recognition-determinant}
 \det Y=(a+b)^{52}(a-b)^{26}.
\end{equation}
Indeed, use the integer grading defined by $h_\alpha$, or by
$h_\beta+h_{\tau\beta}$, before base change. The matrix $Q$
preserves degrees, while the terms involving $N,T_\alpha,T_-$
raise degree by two and those involving $N^2$ by four.
Thus the diagonal part is $aI_{78}+bQ$.
Invertibility of $Y$ makes $a+b$ and $a-b$ units.

Write $\Delta_Y(v,w)=Y\{v,w\}-\{Yv,Yw\}$.
In~\eqref{eq:long-recognition-linear}, the $e_{\alpha_1}$- and
$e_{\alpha_6}$-coefficients of $\Delta_Y(h_1,e_{\alpha_1})$ are
$2a(1-a)$ and $2b(1-b)$.
In~\eqref{eq:skew-recognition-linear}, the $h_1$-, $e_{-\beta}$-,
and $e_{-\tau\beta}$-coefficients of
$\Delta_Y(h_1,e_{-\beta})$ are $2u$, $a(1-a)$, and $b(1-b)$.
Hence $a^2=a$, $b^2=b$ in both cases, and $u=0$ in the short-root
case. Now
\[
 (a-b)(a+b-1)=0
\]
gives $a+b=1$, so $a,b$ are complementary idempotents and
$T=aT\times bT$.

First work on $aT$, where $a=1$, $b=0$.
In the long-root case, the $e_{\alpha_2+\alpha_3+\alpha_4}$-coefficient
of $\Delta_Y(h_1,e_{\alpha_3+\alpha_5})$ is $-d$, and the
$h_1$-coefficient of $\Delta_Y(h_3,e_{-\alpha_4})$ is $-u$.
Thus $d=u=0$.
The $e_{\alpha_4}$-coefficient of
$\Delta_Y(h_4,e_{-\alpha_4})$ is then $4(2e-c^2)$.
Thus $e=c^2/2$ and $Y=U_\alpha(c)$.

In the short-root case, the $e_\beta$-coefficient of
$\Delta_Y(h_1,e_{-\beta})$, followed by the $h_5$- and
$e_\beta$-coefficients of $\Delta_Y(h_1,e_{-\tau\beta})$, gives
\[
 2(c^2-2p)=0,\qquad -d=0,\qquad -2(cd+\ell)=0.
\]
Hence $p=c^2/2$, $d=\ell=0$, and $Y=U_-(c)$.
Only powers of two have been cancelled.

\smallskip
\noindent\emph{3. Exclude the graph factor and recover the parameter.}
On $bT$, apply the preceding calculations to $QY$.
The matrix $Q$ preserves the bracket and commutes with
$\mathcal S_N$, and multiplication by $Q$ interchanges the
coefficients $a,b$. Thus $QY=U_\alpha(c')$ or $U_-(c')$ on
this factor. Since both $Y$ and the resulting root element
belong to $\mathcal G(bT)$, a nonzero factor $bT$ would give
$Q\in\mathcal G(bT)$, contrary to
Lemma~\ref{lemma:frame-centralizer}.
Therefore $b=0$, proving the asserted forms.
The reverse inclusions follow from the definition of $\mathcal S_N$.

The root-coordinate formulas make each parameter unique. Finally,
\[
 \sigma(U_\alpha(c))=U_\alpha(\overline c),\qquad
 \sigma(U_-(c))=U_-(-\overline c).
\]
For $Y\in G(T)$, uniqueness gives $c\in T_0$ or $c\in T^-$,
respectively.
\end{proof}

\subsection{Recovery of the coefficient homomorphism}
\label{subsec:short-root-parameters}
\label{subsec:coefficient-recovery}

Root recognition gives additive parameter maps. The $B_2$
relations identify them across root lengths, and a short $A_2$
relation recovers multiplication.

\begin{proposition}[Homomorphisms fixing the frame]
\label{prop:fixed-frame-homomorphisms}
Let $S,T$ be commutative rings with involution and with $2$
invertible. If a homomorphism
\[
 \Psi:E(S)\longrightarrow G(T)
\]
satisfies $\Psi(x_\lambda(1))=x_\lambda(1)$ for every relative
root, then there is a unique unital equivariant ring homomorphism
$\rho:S\to T$ such that
\[
 \Psi(x_\lambda(t))=x_\lambda(\rho(t))
\]
for every admissible parameter. In particular, $\Psi=\rho_*$ and
$\Psi(E(S))\subseteq E(T)$.
If $\Psi:E(S)\to E(T)$ is an isomorphism, then $\rho$ is an
isomorphism of rings with involution.
\end{proposition}

\begin{proof}
The Weyl and diagonal words in the frame are also fixed.
Write $X_\lambda(t)=\Psi(x_\lambda(t))$.

\smallskip
\noindent\emph{Long roots.}
For $s\in S_0$, the element $x_\alpha(s)$ commutes with
$\mathcal S_{N_\alpha}$. Applying $\Psi$ and
Lemma~\ref{lemma:long-root-rigidity} gives
\[
 X_\alpha(s)=x_\alpha(\rho_0(s)),\qquad \rho_0(s)\in T_0.
\]
Faithful parametrization and the root-subgroup law make $\rho_0$
additive, with $\rho_0(1)=1$. Fixed Weyl conjugation gives the
same map on every long root.

\smallskip
\noindent\emph{Short roots.}
For fixed signs $\varepsilon_1,\varepsilon_2\in\{1,-1\}$,
\[
 [x_{e_3-e_4}(s),x_{e_4}(1)]
   =x_{e_3}(\varepsilon_1s)
      x_{e_3+e_4}(\varepsilon_2s)
 \qquad(s\in S_0).
\]
Apply $\Psi$ and compare with the same commutator over $T$.
Cancelling the recognized long-root factor gives
$X_{e_3}(s)=x_{e_3}(\rho_0(s))$ for $s\in S_0$.
For $q\in S^-$, the equality $x_{e_3}(q)=U_-(q)$ and
Lemma~\ref{lemma:skew-root-rigidity} similarly give
\[
 X_{e_3}(q)=x_{e_3}(\rho_-(q)),\qquad \rho_-(q)\in T^-,
\]
where $\rho_-$ is additive. Define
\[
 \rho(t)=\rho_0\!\left(\frac{t+\overline t}{2}\right)
        +\rho_-\!\left(\frac{t-\overline t}{2}\right).
\]
Then $\rho$ is additive and equivariant, with $\rho(1)=1$.
The root-subgroup law gives
$X_{e_3}(t)=x_{e_3}(\rho(t))$ for all $t\in S$.
Weyl transitions replace a parameter by $\pm t$ or
$\pm\overline t$, so the same formula holds on every short root.

\smallskip
\noindent\emph{Multiplication and uniqueness.}
In the short $A_2$ subsystem generated by $\beta_0,e_4$, the
paired absolute coordinates of Section~1 give
\[
 [x_{\beta_0}(t),x_{e_4}(u)]
   =x_{\beta_0+e_4}(\varepsilon tu)
 \qquad(t,u\in S)
\]
for a fixed sign $\varepsilon$.
Applying $\Psi$ and using faithful parametrization yields
$\rho(tu)=\rho(t)\rho(u)$.
Since the root subgroups generate $E(S)$, we have $\Psi=\rho_*$;
one short-root subgroup determines $\rho$ uniquely.
If $\Psi$ is an isomorphism onto $E(T)$, the same argument applied
to $\Psi^{-1}$ gives $\rho':T\to S$.
Evaluation of the two inverse identities on a short-root subgroup
gives $\rho'\rho=1$ and $\rho\rho'=1$.
\end{proof}

The proposition includes $T=B\times B$ with the exchange
involution. A short-root parameter $(a,b)$ is arbitrary: projection
to the first factor gives $x_\beta(a)x_{\tau\beta}(b)$.
The proof uses the diagonal and antidiagonal parts of $(a,b)$,
not a restriction to $a=b$. Long-root parameters are the
fixed pairs $(c,c)$.

\subsection{The adjoint classification}
\label{subsec:multiplicativity-and-completion}

Coefficient recovery now completes the elementary classification;
characteristicity and the trivial frame centralizer then give the
full-group result.

\begin{theorem}
\label{theorem:automorphisms-of-E}
\label{thm:elementary-adjoint-automorphisms}
Every automorphism of $E=E'_{\mathrm{ad},\sigma}(E_6,R)$ has a
unique expression
\[
 \varphi=i_g\circ\rho_*,\qquad
 g\in G,\quad \rho\in\operatorname{Aut}(R,\theta).
\]
Conversely, every such expression defines an automorphism of $E$.
\end{theorem}

\begin{proof}
Proposition~\ref{prop:fixed-frame-homomorphisms} applied to the
normalized automorphism $\psi$ gives $\psi=\rho_*$ for
$\rho\in\operatorname{Aut}(R,\theta)$.
Since $\varphi_1=i_{\widetilde u^{-1}}\circ\varphi$ and
$\psi=i_C\circ\varphi_1$, we obtain
\[
 \varphi=i_{\widetilde uC^{-1}}\circ\rho_*.
\]
Faithful short-root parametrization also shows that the reduction
of $\rho$ is the automorphism $f$ from
Proposition~\ref{prop:reduction}.

If $i_g\circ\rho_*=i_{g'}\circ\rho'_*$, then $g'^{-1}g$
centralizes the frame, because both ring automorphisms fix it.
Lemma~\ref{lemma:frame-centralizer} gives $g=g'$, and evaluation
on a short-root subgroup gives $\rho=\rho'$.
The converse follows from normality of $E$ in $G$ and
coefficientwise functoriality.
\end{proof}

\begin{corollary}
\label{corollary:automorphisms-of-G}
\label{cor:full-adjoint-automorphisms}
Every automorphism of $G=G_{\mathrm{ad},\sigma}(E_6,R)$ has the
same unique form $i_g\circ\rho_*$, with $g\in G$ and
$\rho\in\operatorname{Aut}(R,\theta)$.
\end{corollary}

\begin{proof}
By characteristicity from Section~2, restrict to $E$ and remove
the standard automorphism supplied by
Theorem~\ref{theorem:automorphisms-of-E}.
The remaining automorphism $\Xi$ fixes $E$ pointwise. For
$a\in G$ and $e\in E$, normality gives
\[
 \Xi(a)e\Xi(a)^{-1}=\Xi(aea^{-1})=aea^{-1}.
\]
Thus $a^{-1}\Xi(a)$ centralizes $E$, and hence the frame.
Lemma~\ref{lemma:frame-centralizer} gives $\Xi=1$.
Uniqueness follows by restriction to $E$.
\end{proof}

\subsection{The simply connected form}
\label{subsec:simply-connected-form}

We first show that the full simply connected group is elementary,
then lift the adjoint classification through the central isogeny.
Write $\widetilde E=E_{\mathrm{sc}}$ and
$\widetilde G=G_{\mathrm{sc}}$.
The local torus--elementary factorization gives
\[
 G=T_{\mathrm{ad},\sigma}E,\qquad
 \widetilde G=T_{\mathrm{sc},\sigma}\widetilde E;
\]
see \cite[Proposition~5.8]{GMNormalSubgroups}, and
\cite{AbeLocal} for the split $F_4$ case.

The simple coroots form a basis of the simply connected
cocharacter lattice. Hence every fixed torus element has the form
\[
 \bigl(h_{\alpha_1}(a)h_{\alpha_6}(\overline a)\bigr)
 \bigl(h_{\alpha_2}(b)h_{\alpha_5}(\overline b)\bigr)
 h_{\alpha_3}(c)h_{\alpha_4}(d),
\]
where $a,b\in R^\times$ and $c,d\in R_0^\times$.
Each factor is a relative toral element and is elementary:
for an admissible unit $u$,
\[
\begin{aligned}
 h_{[\alpha]}(u)&=w_{[\alpha]}(u)w_{[\alpha]}(1)^{-1},\\
 w_{[\alpha]}(u)&=x_{[\alpha]}(u)
                   x_{-[\alpha]}(-u^{-1})x_{[\alpha]}(u).
\end{aligned}
\]
For a short class, this is the product of the two commuting
absolute Weyl words with parameters $u,\overline u$.
Therefore
\[
 T_{\mathrm{sc},\sigma}\subseteq\widetilde E,
 \qquad \widetilde G=\widetilde E.
\]
This argument uses the simple-coroot basis, not an inverse
of the $E_6$ Cartan matrix.

Let $\Lambda_r\subseteq\Lambda_{\mathrm{sc}}$ be the absolute
root and weight lattices, with
$\Lambda_{\mathrm{sc}}/\Lambda_r\simeq\mathbb Z/3\mathbb Z$.
For a $\tau$-stable lattice $\Lambda$, put
\[
 \operatorname{Hom}_1(\Lambda,R^\times)
   :=\{\chi\in\operatorname{Hom}(\Lambda,R^\times)
       \mid\chi\circ\tau=\theta\circ\chi\},
\]
where $\operatorname{Hom}$ denotes group homomorphisms.
Thus $\chi(\tau\lambda)=\overline{\chi(\lambda)}$ for every
$\lambda\in\Lambda$.

\begin{theorem}
\label{thm:simply-connected-automorphisms}
Every automorphism of $\widetilde E$ has the form
\[
 \varphi=i_x\circ d_\chi\circ\rho_*,
 \qquad x\in\widetilde E,\quad
 \chi\in\operatorname{Hom}_1(\Lambda_r,R^\times),\quad
 \rho\in\operatorname{Aut}(R,\theta),
\]
where $d_\chi$ is the diagonal automorphism associated with $\chi$.
Conversely, every such composition is an automorphism of
$\widetilde E$.
\end{theorem}

\begin{proof}
\smallskip
\noindent\emph{1. Pass to the adjoint quotient.}
The central isogeny gives a surjection
$\delta:\widetilde E\to E$, since it preserves the parameters
of every relative root generator.
Its kernel is central, and centerlessness of $E$ gives
\[
 \ker\delta=Z(\widetilde E).
\]
Thus $\varphi$ induces an automorphism $i_g\circ\rho_*$ of $E$
by Theorem~\ref{theorem:automorphisms-of-E}.
The fixed torus normalizes $E$, so write
\[
 g=Xh(\chi),\qquad X\in E,\quad
 \chi\in\operatorname{Hom}_1(\Lambda_r,R^\times),
\]
and choose $x\in\widetilde E$ with $\delta(x)=X$.

\smallskip
\noindent\emph{2. Realize the diagonal factor.}
Use the group-algebra construction of
\cite[Section~9]{BuninaMakadiya2A}:
\[
\begin{gathered}
 \varepsilon_\chi:R[\Lambda_r]\longrightarrow R,\qquad
 \varepsilon_\chi(e^\alpha)=\chi(\alpha),\\
 S_\chi=R[\Lambda_{\mathrm{sc}}]
       \otimes_{R[\Lambda_r],\,\varepsilon_\chi}R.
\end{gathered}
\]
Coset representatives for $\Lambda_{\mathrm{sc}}/\Lambda_r$
containing $0$ give an $R$-basis of $S_\chi$.
Thus $S_\chi$ is free of rank three and $R\to S_\chi$ is injective.
The semilinear lattice involution descends to $S_\chi$, and
\[
 \chi'(\lambda)=e^\lambda\otimes1
 \qquad(\lambda\in\Lambda_{\mathrm{sc}})
\]
is an equivariant character extending $\chi$.
The element $h'=h(\chi')\in T_{\mathrm{sc},\sigma}(E_6,S_\chi)$
satisfies
\[
 h'x_{[\alpha]}(t)(h')^{-1}
   =x_{[\alpha]}(\chi(\alpha)t).
\]
The new parameter is admissible over $R$: for a long root,
$\chi(\alpha)\in R_0^\times$, and for a short root,
$\chi(\tau\alpha)=\overline{\chi(\alpha)}$.
The same holds for conjugation by $(h')^{-1}$.
Through the faithful base-change embedding, these conjugations
therefore define an automorphism $d_\chi$ of $\widetilde E$,
whose action depends only on $\chi$.
Neither \'etaleness of $S_\chi/R$ nor invertibility of $3$ is needed.

\smallskip
\noindent\emph{3. Eliminate the central correction.}
The automorphism
$\varphi^{\mathrm{st}}=i_x\circ d_\chi\circ\rho_*$
induces the same automorphism of $E$ as $\varphi$.
Consequently,
\[
 (\varphi^{\mathrm{st}})^{-1}\varphi(y)=yz(y),
 \qquad z(y)\in Z(\widetilde E),
\]
where centrality makes $z:\widetilde E\to Z(\widetilde E)$ a
group homomorphism.

The group $\widetilde E$ is perfect: long and short $A_2$
subsystems express root elements of the corresponding length as
commutators with parameters $\pm t,1$ and constants $\pm1$.
Here $t\in R_0$ for long roots and $t\in R$ for short roots.
Weyl conjugation reaches every root, so every elementary generator
is a commutator.
Thus $z=1$ and $\varphi=\varphi^{\mathrm{st}}$.
Strictly inner and ring automorphisms, together with the construction
of $d_\chi$ above, also prove the converse.
\end{proof}

\begin{corollary}
\label{cor:full-simply-connected-automorphisms}
Every automorphism of $\widetilde G$ is a product of strictly
inner, diagonal, and ring automorphisms.
\end{corollary}

\begin{proof}
Use $\widetilde G=\widetilde E$ and
Theorem~\ref{thm:simply-connected-automorphisms}.
\end{proof}

This completes Theorem~A for both isogeny types.
In the adjoint case, uniqueness gives
\[
 \operatorname{Aut}E\cong\operatorname{Aut}G
   \cong G\rtimes\operatorname{Aut}(R,\theta),
\]
where $\rho$ acts on $G$ coefficientwise and $(g,\rho)$ corresponds
to $i_g\circ\rho_*$.

\subsection{Relative normalization with a local target}
\label{subsec:relative-local-normalization}

The same normalization applies to homomorphisms with a possibly
nonlocal source and a local target, provided that their residual
action is standard as specified below.

\begin{proposition}
\label{prop:relative-local-normalization}
Let $S$ be a commutative ring with involution and $T$ a local
commutative ring with involution, with $2$ invertible in both.
Put
\[
 \mathfrak n=\operatorname{Jac}(T),\qquad
 q:T\longrightarrow\overline T:=T/\mathfrak n,
\]
and give $\overline T$ the induced involution.
Suppose that a homomorphism $\Psi:E(S)\to E(T)$ satisfies
\[
 q_*\circ\Psi=i_h\circ(\nu_0)_*,
 \qquad h\in G(\overline T),
\]
for a unital equivariant ring homomorphism $\nu_0:S\to\overline T$.
Then there exist a unique element $g\in G(T)$ and a unique
unital equivariant ring homomorphism $\rho:S\to T$ such that
\[
 q_*(g)=h,\qquad q\circ\rho=\nu_0,\qquad
 \Psi=i_g\circ\rho_*.
\]
The element $g$ is already uniquely determined by
\[
 g\,x_{\lambda,T}(1)\,g^{-1}
   =\Psi(x_{\lambda,S}(1))
 \qquad(\lambda\in\Phi_\rho),
\]
where the second subscript specifies the coefficient ring.
Neither $\Psi$ nor $\nu_0$ is assumed injective or surjective.
\end{proposition}

\begin{proof}
By Lemma~\ref{lemma:surjective-reduction-G}, lift $h$ to
$\widetilde h\in G(T)$ and put
$\Psi_1=i_{\widetilde h^{-1}}\circ\Psi$.
Then $q_*\circ\Psi_1=(\nu_0)_*$, so all frame images have standard
reductions.

The common integral Weyl model from
Subsection~\ref{subsec:initial-weyl-family} and the homomorphism
versions of the normalizations in Sections~3--4 give
\[
\begin{gathered}
 C_0\in\operatorname{GL}_{78}(T)^\sigma,\qquad
 C_0\equiv I_{78}\pmod{M_{78}(\mathfrak n)},\\
 C_0\Psi_1(x_{\lambda,S}(1))C_0^{-1}=x_{\lambda,T}(1)
 \qquad(\lambda\in\Phi_\rho).
\end{gathered}
\]
The matrices $\Psi_1(x_{\lambda,S}(1))$ belong to $E(T)$ and
preserve the bracket. Proposition~\ref{prop:descent-of-correction}
therefore replaces $C_0$ by $C\in G(T)$ with the same congruence
and frame conjugacies.

Set $g=\widetilde hC^{-1}$.
Then $q_*(g)=h$, and
$\Psi^0:=i_{g^{-1}}\circ\Psi=i_C\circ\Psi_1$ fixes the frame.
Proposition~\ref{prop:fixed-frame-homomorphisms} gives
$\Psi^0=\rho_*$ for a unique unital equivariant ring homomorphism
$\rho:S\to T$.
Since $q_*\circ\Psi^0=(\nu_0)_*$, evaluation on a short-root
subgroup gives $q\circ\rho=\nu_0$.

Any other factorization has the same frame conjugacies, because
unital coefficient homomorphisms fix the frame.
Lemma~\ref{lemma:frame-centralizer} makes its conjugating element
equal to $g$; faithful short-root parametrization then makes its
coefficient homomorphism equal to $\rho$.
\end{proof}

For every unital equivariant ring homomorphism $f:T\to U$,
\[
 f_*\circ\Psi=i_{f_*(g)}\circ(f\circ\rho)_*.
\]
By Lemma~\ref{lemma:frame-centralizer}, $f_*(g)$ remains the unique
element of $G(U)$ giving these frame conjugacies, without a
locality assumption on $U$.

\section*{Computational verification}

The archive \texttt{twisted\_e6\_verification.zip} contains the
source code, exact systems, model matrices, selected rows and
columns, and certificate minors for the computational parts
of the proof. The accompanying \texttt{verification\_report.pdf}
summarizes the results; the machine-readable summary is included
in the archive as \texttt{output/verification\_results.json}.
The calculations are organized as follows.
\begin{center}
\begin{tabular}{@{}p{.48\linewidth}p{.48\linewidth}@{}}
\hline
Calculation & Directory under \texttt{sections/} \\ \hline
Complete short-root normalization
(Lemma~\ref{lemma:full-short-root-normalization})
  & \texttt{section4\_complete\_short\_root/} \\
Invariant alternating tensors
(Lemma~\ref{lemma:invariant-alternating-tensors})
  & \texttt{section5\_invariant\_tensors/} \\
Long- and skew-root recognition
(Lemma~\ref{lemma:long-root-rigidity})
  & \texttt{section6\_parameter\_rigidity/} \\ \hline
\end{tabular}
\end{center}

After installing the requirements listed in the archive, run
\begin{center}
\texttt{python run\_all.py}
\end{center}
from its root directory. This checks the file hashes and stored
certificates, regenerates all systems in a temporary copy, and
compares the regenerated mathematical data with the supplied files.
The published files are not overwritten.
All proof calculations use exact arithmetic over $\mathbb Z[1/2]$.
The unit-minor and model-vector arguments in the corresponding
proofs justify their use over the stated coefficient rings.
In particular, the long-root relation $3u=0$ is retained until
bracket preservation gives $u=0$; no division by $3$ is used.

Lemma~\ref{lemma:bracket-stabilizer-tangent} has a direct proof
in the text. The additional computation in
\texttt{sections/section5\_bracket\_stabilizer/}
is a consistency check, not a prerequisite for that proof.
The archive also checks the Chevalley sign conventions, opposite-root
normalizations, Jacobi identities, and standard Weyl words.


\medskip

\section*{Declaration on the use of generative AI}

The authors used ChatGPT (OpenAI) to assist with organizing and
revising the exposition, exploring proof simplifications, checking
intermediate calculations, and developing verification code.
The authors take full responsibility for the mathematical content,
computational certificates, and final manuscript.



\begin{thebibliography}{99}

\bibitem{AbeLocal}
E.~Abe,
\emph{Chevalley groups over local rings},
Tohoku Math. J. (2) \textbf{21} (1969), 474--494.

\bibitem{AbeTwisted}
E.~Abe,
\emph{Coverings of twisted Chevalley groups over commutative rings},
Sci. Rep. Tokyo Kyoiku Daigaku Sect.~A \textbf{13} (1977), 194--218.

\bibitem{Borel}
A.~Borel,
\emph{Linear algebraic groups},
2nd ed., Graduate Texts in Mathematics, vol.~126,
Springer-Verlag, New York, 1991.

\bibitem{BuninaLocalADE}
E.~I.~Bunina,
\emph{Automorphisms of elementary adjoint Chevalley groups of types
$A_l,D_l,E_l$ over local rings with $1/2$},
Algebra and Logic \textbf{48} (2009), no.~4, 250--267.

\bibitem{Bunina2012}
E.~I.~Bunina,
\emph{Automorphisms of Chevalley groups of different types over
commutative rings},
J. Algebra \textbf{355} (2012), 154--170.

\bibitem{Bunina2024}
E.~I.~Bunina,
\emph{Automorphisms of Chevalley groups over commutative rings},
Comm. Algebra \textbf{52} (2024), no.~6, 2313--2327.

\bibitem{BuninaMakadiya2A}
E.~I.~Bunina and D.~H.~Makadiya,
\emph{Automorphisms of twisted Chevalley groups of type
${}^{2}A_\ell$ $(\ell\geq5)$ over local rings},
arXiv:2607.22284 (2026).

\bibitem{BuninaMakadiya2A34}
E.~I.~Bunina and D.~H.~Makadiya,
\emph{Automorphisms of twisted Chevalley groups of type
${}^{2}A_3$ and ${}^{2}A_4$ over local rings},
arXiv:2609.02617 (2026).

\bibitem{BuninaMakadiya2D}
E.~I.~Bunina and D.~H.~Makadiya,
\emph{Automorphisms of twisted Chevalley groups of type
${}^{2}D_\ell$ $(\ell\geq4)$ over local rings},
arXiv:2608.02312 (2026).

\bibitem{Chevalley}
C.~Chevalley,
\emph{Certains sch\'emas de groupes semi-simples},
S\'eminaire Bourbaki \textbf{6} (1960/61),
Exp. No.~219, 219--234.

\bibitem{GMNormalSubgroups}
S.~M.~Garge and D.~H.~Makadiya,
\emph{On normal subgroups of twisted Chevalley groups over
commutative rings},
with an appendix by P.~Gvozdevsky,
J. Algebra \textbf{688} (2026), 587--659.

\bibitem{Humphreys1969}
J.~E.~Humphreys,
\emph{On the automorphisms of infinite Chevalley groups},
Canad. J. Math. \textbf{21} (1969), 908--911.

\bibitem{HumphreysLie}
J.~E.~Humphreys,
\emph{Introduction to Lie algebras and representation theory},
Graduate Texts in Mathematics, vol.~9,
Springer-Verlag, New York--Berlin, 1972.

\bibitem{StacksProject}
The Stacks Project Authors,
\emph{The Stacks Project},
2026,
\url{https://stacks.math.columbia.edu}.

\bibitem{SteinbergLectures}
R.~Steinberg,
\emph{Lectures on Chevalley groups},
University Lecture Series, vol.~66,
American Mathematical Society, Providence, RI, 2016.

\bibitem{SuzukiLocalTwisted}
K.~Suzuki,
\emph{On normal subgroups of twisted Chevalley groups over local rings},
Sci. Rep. Tokyo Kyoiku Daigaku Sect.~A \textbf{13} (1977), 238--249.

\bibitem{SuzukiNormality}
K.~Suzuki,
\emph{Normality of the elementary subgroups of twisted Chevalley
groups over commutative rings},
J. Algebra \textbf{175} (1995), 526--536.

\bibitem{TitsWeyl}
J.~Tits,
\emph{Normalisateurs de tores. I. Groupes de Coxeter \'etendus},
J. Algebra \textbf{4} (1966), 96--116.

\bibitem{PlotkinVavilov}
N.~A.~Vavilov and E.~B.~Plotkin,
\emph{Chevalley groups over commutative rings. I. Elementary
calculations},
Acta Appl. Math. \textbf{45} (1996), 73--113.

\end{thebibliography}
\end{document}